\documentclass[11pt, a4paper]{article}%
\usepackage{eurosym}
\usepackage{amsmath}
\usepackage{amsbsy}
\usepackage{amsfonts}
\usepackage{bm}
\usepackage{amsfonts, graphicx, verbatim, amsmath,amssymb}
\usepackage[hidelinks]{hyperref}
\usepackage{color}
\usepackage{amssymb}
\usepackage{graphicx}
\usepackage{enumitem}
\usepackage{stmaryrd}
\usepackage{cancel}

\newtheorem{Condition}{Condition}[section]

\providecommand{\U}[1]{\protect\rule{.1in}{.1in}}
\providecommand{\U}[1]{\protect\rule{.1in}{.1in}}
\usepackage[left=2cm,right=2cm, top=2cm, bottom=2cm]{geometry}

\DeclareMathOperator{\argmin}{argmin}

\newcommand{\realset}[1]{\mathbb{R}^{#1}}

\DeclareMathOperator{\Var}{Var}

\newcommand{\asympConMat}{M_n}
\newcommand{\asympConMatapprox}{M_{n, N}}
\newcommand{\asympConMatapproxNn}{M_{n, N_n}}
\newcommand{\thetaML}{\widehat{\theta}_n}
\newcommand{\thetaMLapprox}{\widehat{\theta}_{n, N}}
\newcommand{\thetaMLapproxNn}{\widehat{\theta}_{n, N_n}}

\newcommand{\functionSpace}{\mathcal{H}}

\newcommand{\negloglikeFun}{\mathcal{L}}

\newcommand{\cov}{\operatorname{cov}}
\newcommand{\tr}{\operatorname{tr}}
\newcommand{\expect}{\mathbb{E}}

\renewcommand{\d}[1]{\textrm{d}{#1}}

\newtheorem{theorem}{Theorem}[section]

\newtheorem{corollary}[theorem]{Corollary}

\newtheorem{example}[theorem]{Example}

\newtheorem{lemma}[theorem]{Lemma}

\newenvironment{proof}[1][Proof]{\textbf{#1.} }{\ \rule{0.5em}{0.5em}\newline}
\def\build#1_#2^#3{\mathrel{\mathop{\kern 0pt#1}\limits_{#2}^{#3}}}
\newcommand{\converge}[3]{\build\hbox to
	15mm{\rightarrowfill}_{#1\rightarrow #2}^{\hbox{\scriptsize #3}}}
\newcommand{\convergedef}[1]{\overset{{\hbox{\scriptsize #1}}}{\hbox to 15mm{\rightarrowfill}}}

\begin{document}
\title{Asymptotic analysis for covariance parameter estimation of Gaussian processes with functional inputs}
\author{{Lucas Reding}\thanks{Corresponding author}, {Andrés Felipe L\'opez-Lopera} and François Bachoc}
\date{\vspace{-3em}}
\maketitle

Univ. Polytechnique Hauts-de-France, INSA Hauts-de-France, CERAMATHS, F - 59313 Valenciennes, France

Email: lucas.reding@uphf.fr; andres.lopezlopera@uphf.fr

Institut de Math\'ematiques de Toulouse, Universit\'e Paul Sabatier, F-31062 Toulouse, France.

Email: francois.bachoc@math.univ-toulouse.fr

\begin{abstract}
    We consider covariance parameter estimation for Gaussian processes with functional inputs. From an increasing-domain asymptotics perspective, we prove the asymptotic consistency and normality of the maximum likelihood estimator. We extend these theoretical guarantees to encompass scenarios accounting for approximation errors in the inputs, which allows robustness of practical implementations relying on conventional sampling methods or projections onto a functional basis. Loosely speaking, both consistency and normality hold when the approximation error becomes negligible, a condition that is often achieved as the number of samples or basis functions becomes large. These later asymptotic properties are illustrated through analytical examples, including one that covers the case of non-randomly perturbed grids, as well as several numerical illustrations.
\end{abstract}


\section{Introduction}
\label{sec:intro}

Gaussian processes (GPs) represent a widely used class of stochastic processes in geostatistics~\cite{stein1999interpolation}, machine learning~\cite{Hemachandran2022bayesianGPML,Rasmussen2005GP}, computer experiments~\cite{Bachoc2017GPCompExp,Binois2022SurveyGPsBO}, and engineering~\cite{Camps2016SurveyGPsDataAnalysis,Forrester2008SurrogatesEng}. A GP forms a non-parametric probabilistic framework wherein a prior belief, i.e. a distribution, is placed over a function that is either unknown or costly to evaluate. More precisely, the latter function is modeled as a realization of the (conditional) GP that interpolates some observation points. For instance, in computer experiments, if the latent function represents a computer code, then the observation points correspond to evaluations of this code~\cite{Bachoc2017GPCompExp}. In regression tasks, the predictive GP mean serves as an optimal linear unbiased predictor, and the predictive GP variance quantifies the expected square error associated with this predictor~\cite{Rasmussen2005GP,stein1999interpolation}.
\\
\\
The covariance function of the GP, also known as the kernel, plays an important role in defining correlations between values of the process at different input points. For instance, in geostatistics, this function seeks to quantify the similarity between data points, indicating that nearby spatial inputs are more likely to yield similar observation values. It is commonly assumed that the covariance function belongs to a parametric family of symmetric and positive semidefinite functions. This family is often defined based on prior beliefs (e.g. regularity, stationarity, periodicity) of the behavior of the function that is being modeled~\cite{Genton2001Kernels, Paciorek2004NonStCov}. Therefore, the choice of a covariance function boils down to the estimation of the associated parameters of the parametric family, namely the covariance parameters. The two main approaches for covariance parameter estimation are based on maximum likelihood (ML) or cross-validation~\cite{Bachoc2013_GPEstimators, Bachoc2021AsymptAnalMLEGPs,stein1999interpolation}, with the former being the most common one in both theoretical and practical studies. 
\\
\\
Asymptotic guarantees, such as the consistency and normality, of ML estimators for GPs indexed by vectors in $\mathbb{R}^d$, with $d \geq 1$, have been further investigated in~\cite{Bachoc2014AsymptAnalGPs,furrer2023asymptotic,ShaKau2013ParamEstim, stein1999interpolation, Zhang04Inconsistent}. These guarantees are primarily explored within two frameworks: increasing-domain asymptotics and fixed-domain asymptotics. Within increasing-domain asymptotics, the smallest distance between two inputs is lower bounded. In contrast, fixed-domain asymptotics involves a sequence (or a triangular array) of data points that becomes densely distributed in a fixed, bounded subset of $\mathbb{R}^d$. While results from increasing-domain asymptotics are broadly applicable to various covariance functions, results from fixed-domain asymptotics require specific derivation on a case-by-case basis. This implies that for fixed-domain asymptotics, the analysis must be customized to the characteristics of each covariance function individually, limiting the generalization of the results across a broader range of functions.
\\
\\
Recent works have extended the implementation of GPs to non-Euclidean inputs, such as distributions~\cite{Bachoc2023GPDist_OptTransport,Bachoc2020GaussFieldSymmGroup,Bachoc2018_GPdistInputs,Bachoc2020GPs_multidimDist,Szabo2016DistRegression}, functions~\cite{Betancourt2020fGPs,Lopezlopera2021multioutputGP,Muehlenstaedt2017CompExpFunInputs}, categorical or integer-valued variables~\cite{Garrido2020GP_catintInputs,Roustant2020_GPsCatInputs}, and graphs~\cite{Borovitskiy2021GPGraphs,Zhi2023GPGraph_SpectKernel}. There, adapted similarity distances have been considered to construct valid covariance functions. For instance, for distribution and functional inputs, similarity distances based on the Wasserstein distance~\cite{Bachoc2018_GPdistInputs} and the $L_2$-norm~\cite{Lopezlopera2021multioutputGP,Muehlenstaedt2017CompExpFunInputs} (respectively) have been considered. The references~\cite{Bachoc2023GPDist_OptTransport,Bachoc2020GPs_multidimDist} address multidimensional distributions via optimal transport and Hilbertian embedding. In particular,~\cite{Bachoc2020GPs_multidimDist} uses the Wasserstein barycenters to define general covariance functions using optimal transportation maps. Despite significant numerical progress in these models, establishing theoretical guarantees certifying their robustness remains challenging. Among the few available theoretical guarantees, for the case of 1D distribution inputs, the consistency and normality of the ML estimator under increasing-domain asymptotics have been showed in~\cite{Bachoc2018_GPdistInputs}.
\\
\\
Building on the promising numerical outcomes demonstrated by GPs with functional inputs for tasks such as surrogate modeling~\cite{Betancourt2020fGPs, Lopezlopera2021multioutputGP,Muehlenstaedt2017CompExpFunInputs}, sampling and sensitivity analysis~\cite{Nanty2016SimuFunStochInputs}, our focus in this paper is to establish asymptotic guarantees for this class of models. More precisely, we focus on the consistency and normality of the ML estimator from an increasing-domain asymptotic perspective. In a general context, our analysis considers stationary covariance functions on some Hilbert space. We first prove both the asymptotic consistency and normality of the ML estimator when the functional form of the inputs is assumed to be fully known. However, since this latter assumption may not be realistic in numerical implementations due to discretization or the availability of only partial observations~\cite{Ramsay2005functional}, we also extend the previous asymptotic results to scenarios where errors (e.g. resulting from sampling or linear projections) are considered. These additional guarantees enable the certification of the robustness of numerical implementations of the GP model. Loosely speaking, both the consistency and normality of the ML estimator hold when the approximation error becomes negligible, a condition that is often achieved as the number of samples or basis functions tends to be infinite. 
\\
\\
Throughout this paper, we provide theoretical and numerical examples to illustrate the proposed conditions and theorems. More precisely, we examine an example that covers the case of nonrandomly perturbed grids, showcasing that the sufficient conditions outlined in this work can be easily met. Additionally, we illustrate the resulting theorems with various analytical and numerical examples. The Python codes related to the numerical implementations are available in the Github repository: \url{https://github.com/anfelopera/fGPs}. 
\\
\\
This paper is structured as follows. Sections~\ref{sec:generalities+GPs} recalls essential definitions and properties related to GPs with functional inputs, as well as the ML estimator. In Sections~\ref{sec:conditions} and~\ref{sec:asympProperties}, the focus shifts to introducing the necessary conditions, lemmas and theorems establishing the asymptotic consistency and normality of the ML estimator. There, we neglect the approximation errors over the functional inputs. Section~\ref{sec:asympPropertiesApprox} extends the previous asymptotic results to account for the approximation errors. Section~\ref{sec:examplesAsympPropertiesApprox} provides analytical examples and numerical illustrations. Finally, Section~\ref{sec:conclusions} summarizes the conclusions drawn from our analysis and outlines potential avenues for future research.

Variations of the outlined conditions for isotropic covariance functions, and the proofs of the lemmas are provided in Appendices~\ref{app:CondIsotropy},~\ref{app:proofLemmas} and~\ref{app:proofLemmasApprox}.

\section{Gaussian processes with functional inputs}
\label{sec:generalities+GPs}

\subsection{Generalities}
	\label{sec:generalities+GPs:subsec:generalities}
    Before introducing the GP framework that will be considered in this paper, we first discuss some generalities related to the functional context. Throughout the following sections, we consider a separable Hilbert space $\functionSpace$, equipped with an inner product denoted $\langle \cdot, \cdot \rangle$. Here, $\functionSpace$ does not need to be real nor of finite dimension. We define the norm derived from the inner product as $\|\cdot\|$. In other words, for any $f \in \functionSpace$, we have
	\begin{equation*}
		%\normL(f,g) = 
		\|f\| = \sqrt{\langle f, f \rangle} < \infty.
    \end{equation*}
    We denote, for any $k \ge 1$, $\langle \cdot, \cdot \rangle_{\mathbb{R}^k}$, the standard Euclidean inner product on $\mathbb{R}^k$, and $L^2_k([0,1])$ the space of $\mathbb{R}^k$-valued, square integrable functions defined on $[0,1]$. Note that if $k = 1$, we will simply write $L^2_1([0,1]) := L^2([0,1])$. In the following, although final applications in this paper consider $\functionSpace = L^2([0,1])$, the asymptotic analysis is valid for any Hilbert space $\functionSpace$. Next, we provide some examples of Hilbert spaces together with some corresponding inner products.
    
    \noindent
    \underline{Euclidean:} $\functionSpace = \mathbb{R}^d$
     with $d \ge 1$ and endowed with the standard Euclidean inner product. This corresponds to standard GPs indexed by vectors \cite{Rasmussen2005GP}.
    
    \noindent 
    \underline{Univariate functional with multidimensional parameter:}
    $\functionSpace = \mathbb{R}^d \times L^2([0,1])$\\ with $d \ge 1$ and endowed with the inner product
     \[
       \forall (x,f), (x',g) \in \functionSpace, \qquad \langle (x,f), (x',g) \rangle = \langle x, x' \rangle_{\mathbb{R}^d} + \int_{0}^1{f(t)g(t)\d{t}}.
     \]
%    This can, for instance, model a computer with $d$ scalar inputs and one scalar-value functional input \cite{Betancourt2020fGPs}.
     \noindent
     \underline{Multivariate functional with multidimensional parameter:} $\functionSpace = \mathbb{R}^d \times L^2_n([0,1])$ with $n,d \ge 1$ and endowed with the inner product
     \[
       \forall (x,f), (x',g) \in \functionSpace, \qquad \langle (x,f), (x',g) \rangle = \langle x, x' \rangle_{\mathbb{R}^d} + \int_{0}^1{\langle f(t),g(t) \rangle_{\mathbb{R}^n}\d{t}}.
     \]
     The previous two examples can, for instance, model a computer code with $d$ scalar inputs and one scalar-valued or vector-valued (respectively) functional input \cite{Betancourt2020fGPs}.
    
    \noindent
    \underline{Complex functional with multidimensional parameter:} $\functionSpace = \mathbb{C}^d \times L^2([0,1];\mathbb{C})$ with $d \ge 1$, and $L^2([0,1];\mathbb{C})$ the space of $\mathbb{C}$-valued, square integrable functions defined on $[0,1]$. Denote by $(\cdot|\cdot)_{\mathbb{C}^d}$ a complex inner product on $\mathbb{C}^d$. Endowing $\functionSpace$ with the following inner product
     \[
       \forall (x,f), (x',g) \in \functionSpace, \qquad \langle (x,f), (x',g) \rangle = (x|x')_{\mathbb{C}^d} + \int_{0}^1{f(t)\overline{g(t)}\d{t}},
     \]
     yields an example of a complex, separable Hilbert space.

\subsection{Gaussian processes}
    \label{sec:generalities+GPs:subsec:GPs}
    % We first introduce the framework we will be working with. We also recall some important properties of the Gaussian fields we are interested in. In everything that follows we consider a Hilbert space $\functionSpace$  and we denote by $\langle \cdot, \cdot \rangle$ its inner product. Over this space, we also write $\|\cdot\|$ for the norm derived from this inner product. In other words, for any $f \in \functionSpace$, we have

    Let us consider a GP $Y$ indexed by $\functionSpace$ and defined over an underlying probability space $(\Omega, \mathcal{F}, \mathbb{P})$, with zero-mean function and unknown covariance function (kernel) $K_0$. We assume that $K_0$ belongs to a parametric family of the form
    \begin{equation}
	\label{eq:kernel}
	\left\{K_\theta; \theta \in \Theta \right\},
    \end{equation}
    with $\Theta \subset \realset{p}$ compact and $\theta$ in the interior of $\Theta$. The parameter $\theta = (\theta_1, \ldots, \theta_p)$ is referred to as the covariance parameter of $K_\theta$. In particular, we denote $K_0 := K_{\theta_0}$ for some unknown covariance parameter $\theta_0 \in \Theta$. For instance, it is customary to consider the squared exponential (SE) covariance structure (see, e.g., \cite{Bachoc2018_GPdistInputs,Rasmussen2005GP}):
    \begin{equation}
        \left\{
	K_{\theta} = \theta_1 e^{-\frac{\|f-g\|^2}{2\theta_2^2}}, 
	\theta = (\theta_1, \theta_2) \in C \times C'
	\right\},
	\label{eq:statkernel}
    \end{equation}
    for $f, g \in \functionSpace$ and $C, C'$ compact sets in $(0, \infty)$. For the kernel structure in~\eqref{eq:statkernel}, $(\theta_1, \theta_2)$ corresponds to the variance and the length-scale parameters (respectively) of $K_{\theta}$. While $\theta_1$ can be seen as a scale parameter of the output, $\theta_2$ can be viewed as a scale parameter for the input variable.

    The parameters $\theta$ are typically estimated from a training dataset which consists of tuples of inputs and outputs. In our case, this dataset is given by $(f_i, y_i)_{i=1,\ldots, n}$, where $f_i$ can represent a deterministic function (i.e. $f_i \in \functionSpace$ with $\functionSpace = L^2([0,1])$) and $y_i$ represents the corresponding output when considering $f_i$. For instance, in a coastal flooding context, the outputs may correspond to the outcome of a computer code modeling the water height of a flood event as a function of the tide, with the latter being a time-varying function~\cite{Betancourt2020fGPs,Lopezlopera2021multioutputGP}. The underlying mapping function relating $f_i$ and $y_i$ can be then modeled as a realization of a GP $Y$ such that $Y(f_i) = y_i$ for all $i = 1, \ldots, n$. Due to the properties of Gaussian distributions, the conditional process $Y$ that interpolates the observations is also GP-distributed~\cite{Rasmussen2005GP,stein1999interpolation}. In this paper, we are interested in the case where the underlying mapping function indeed comes from a GP $Y$ with covariance function $K_{\theta_0}$.

    \paragraph{Covariance parameter estimation via maximum likelihood.} 
    %In a general GP context, several methods have been proposed for constructing an estimator of the covariance parameters $\theta_0$, including maximum likelihood (ML) and cross-validation~\cite{Bachoc2021AsymptAnalMLEGPs,stein1999interpolation}.
    Here, we focus on ML estimation (MLE), a widely used method that has received great attention from a theoretical point of view~\cite{Bachoc2021AsymptAnalMLEGPs,furrer2023asymptotic,stein1999interpolation}. MLE involves maximizing the Gaussian likelihood, which is equivalent to minimizing the modified negative log Gaussian likelihood:
	\begin{equation}
		\negloglikeFun_{n}(\theta) = \frac{1}{n} \ln|R_\theta| + \frac{1}{n} y^\top R_\theta^{-1} y,
		\label{eq:negLogGaussianLike}
	\end{equation}
%\negloglikeFun_{n}(\theta) = \frac{1}{n} \ln|R_\theta| + \frac{1}{n} y^\top R_\theta^{-1} y  + \ln(2\pi),
    with $y = [y_1, \ldots, y_n]^\top$, $R_\theta = R(\theta) := (K_\theta(f_i, f_j))_{1 \leq i,j \leq n}$ and $|R_\theta|$ the determinant of $R_\theta$. Then, the ML estimator is $\thetaML 
    %:= \thetaML((f_1, y_1), \ldots, (f_n, y_n))
    \in \argmin_{\theta \in \Theta} \negloglikeFun_{n}(\theta)$.
    
    %The Gaussian likelihood of the vector of observations $y = [y_1, \ldots, y_n]^\top$ is given by
	%\begin{equation}
	%	\likeFun_\theta = \frac{1}{|R_\theta|^{\frac{1}{2}} (2\pi)^{\frac{n}{2}}} e^{-\frac{1}{2} y^\top R_\theta^{-1} y},
	%	\label{eq:GaussianLike}
	%\end{equation}
	%with $R_\theta = R(\theta)  := (K_\theta(f_i, f_j))_{1 \leq i,j \leq n}$ and $|R_\theta|$ the determinant of $R_\theta$. Maximum Likelihood is based on maximizing~\eqref{eq:GaussianLike}, an optimization task that is equivalent 

    \paragraph{Gaussian process regression.} In regression tasks, the value $Y(f)$, for any input $f \in \functionSpace$, can be predicted by plugging $\thetaML$ into the conditional expectation expression (also known as the posterior mean) for GPs~\cite{Rasmussen2005GP,stein1999interpolation}. More precisely, denoting $\widehat{Y}_{{\theta_0}}(f)$ for the conditional expectation of $Y(f)$ given $Y(f_1) = y_1, \ldots, Y(f_n) = y_n$, then we have
	\begin{equation*}
		\widehat{Y}_{{\theta_0}}(f) 
        := 
        \expect(Y(f) | Y(f_1) = y_1, \ldots, Y(f_n) = y_n)
        = r_{{\theta_0}}^\top(f) R_{{\theta_0}}^{-1} y,
	\end{equation*}
	with $r_{{\theta}}(f) = [K_{{\theta}}(f,f_1), \ldots, K_{{\theta}}(f,f_n)]^\top$ and $R_\theta = (K_\theta(f_i, f_j))_{1 \leq i,j \leq n}$ for any $\theta \in \Theta$. Given that the parameter $\theta_0$ is unknown in practice, the GP prediction at $f$ is thus given by
    \begin{equation*}
		\widehat{Y}_{{\thetaML}}(f) 
        := r_{{\thetaML}}^\top(f) R_{{\thetaML}}^{-1} y.
		\label{eq:pred}
	\end{equation*}
 Furthermore, the expected squared error of this estimate is obtained by computing the conditional variance at $f$
    \begin{equation*}
		\Var(Y(f) | Y(f_1) = y_1, \ldots, Y(f_n) = y_n) = K_{{\theta_0}}(f,f) - r_{{\theta_0}}^\top(f) R_{{\theta_0}}^{-1} r_{{\theta_0}}(f).
	\end{equation*}
    and replacing the unknown parameter $\theta_0$ by its estimate $\thetaML$, yielding the following estimator
    \begin{equation*}
         K_{{\thetaML}}(f,f) - r_{{\thetaML}}^\top(f) R_{{\thetaML}}^{-1} r_{{\thetaML}}(f).
         \label{eq:predVar}
    \end{equation*}
    \section{Sufficient conditions}
    \label{sec:conditions}
    In the increasing-domain asymptotic analysis, we need to consider a deterministic triangular array $\{f_1^{n}, \ldots, f_{n}^{n}\}_{n \in \mathbb{N}^*}$ of elements in $\functionSpace$. We require the input sequences to satisfy the following separability hypothesis.
    \begin{Condition}
	\label{cond:distant}	
	Suppose that there exists $\Delta > 0$ such that for all $n \in \mathbb{N}^*$ and for all $1 \leq i \ne j \leq n$, we have $\|f_i^n-f_j^n\| \geq 2 \Delta$.
    \end{Condition}
    When $\functionSpace = \realset{d}, d \geq 1$, Condition~\ref{cond:distant}	is standard in the increasing-domain asymptotics literature. In the following, to ease the notation, we will simply write $f_i = f_i^{n}$ for all $1 \le i \le n$ and whenever $n$ is fixed.
    
    Condition~\ref{cond:orthormalBasis} imposes a form of uniform control over the projections of the input elements over a fixed orthonormal basis. More precisely, we require that the projections whose orders exceed a certain value are uniformly (with respect to the input elements) bounded in such a way that the series of the squared bounds is convergent. This condition will be needed to use approximation arguments and derive both the asymptotic consistency as well as the asymptotic normality of the estimator.  
    \begin{Condition}
	\label{cond:orthormalBasis}
	For
 $n \in \mathbb{N}^*$ and $1 \le i \le n$, we denote 	\begin{equation*}
		f_{i,d} = \sum_{j=1}^d \langle f_i, e_j\rangle e_j,
	\end{equation*}
	with $(e_j)_{j \in \mathbb{N}^*}$ an orthonormal basis of $\mathcal{H}$. We assume there is a $J_0$ and a sequence $(s_j)_{j \ge J_0}$ of non-negative numbers such that for all $1 \leq i \leq n$, $|\langle f_i, e_j \rangle| \leq s_j$, for $j \geq J_0$, with $\sum_{j=J_0}^{\infty} s_j^2 < \infty$. 
\end{Condition}

Note that Condition~\ref{cond:orthormalBasis} with $J_0=1$ and Condition~\ref{cond:distant} are incompatible. Indeed, if they hold,
%	\begin{align*}
	%		f_{i} &= \sum_{k=1}^\infty \langle f_{i}, e_k\rangle e_k,
	%		\\		
	%		f_{j} &= \sum_{k=1}^\infty \langle f_{j}, e_k\rangle e_k.
	%	\end{align*}
%	\begin{equation*}
	%		\|f_{i} - f_{j}\|^2
	%		= \sum_{k=1}^\infty |\langle f_{i} - f_j, e_k\rangle|^2
	%		\leq 4 \sum_{k=1}^\infty s_k^2
	%	\end{equation*}
there exists $d$ such that $\left|\|f_{i,d} - f_{j,d}\| - \|f_{i} - f_{j}\|\right| \leq 2\sqrt{\sum_{k=d+1}^\infty s_k^2} \leq \Delta$. Hence, for $i \ne j$, $\|f_{i,d} - f_{j,d}\| \geq \Delta$. And, since $J_0 = 1$, then
\begin{equation*}
	\|f_{i,d}\|^2 \leq \|f_{i}\|^2 \leq \sum_{k=1}^\infty s_k^2.
\end{equation*}
Identifying the projections $f_{i,d}$ with their coordinates in $\mathbb{R}^d$ through the canonical isometry $\tau_d$, we find that there exists a ball in $\mathbb{R}^d$ which contains an infinite number of $\tau_d(f_{i,d})$ with $\|\tau_d(f_{i,d}) - \tau_d(f_{j,d})\| = \|f_{i,d}-f_{j,d}\| \geq \Delta$ for $i \ne j$, which is a contradiction.
\smallskip

To derive the consistency of the estimator, we will use approximation arguments. Thus, Lemma~\ref{lemma:existd} will play an important role as it guarantees that, under the hypothesis of Theorem~\ref{theo:consistency}, there exists an approximation order large enough to ensure that the approximated input data is well-behaved.

\begin{lemma}
    \label{lemma:existd}
    Suppose that Conditions~\ref{cond:distant} and \ref{cond:orthormalBasis} are satisfied, then there exists $d_0 \ge J_0$ such that for all $1 \leq i < j \leq n$, $\|f_{j,d_0} - f_{i,d_0}\|\geq \Delta$ and $\sum_{k=d_0+1}^\infty s_k^2 < \Delta^2/16$.
\end{lemma}

\begin{proof}[Proof of Lemma~\ref{lemma:existd}]
Let $1 \le i\not = j \le n$ and $d \in \mathbb{N}^*$, then using the triangle inequality as well as Parseval's identity, we get 
	\begin{align*}
		\|f_{j} - f_{i}\|
		&\leq \|f_{j} - f_{j,d}\| +  \|f_{i,d} - f_{i}\| + \|f_{j,d} - f_{i,d}\|
		\\
		&= \Bigg(\sum_{k=d+1}^{\infty}{|\langle f_i,e_k\rangle|^2}\Bigg)^{\frac{1}{2}} + \Bigg(\sum_{k=d+1}^{\infty}{|\langle f_j,e_k\rangle|^2}\Bigg)^{\frac{1}{2}} + \|f_{j,d} - f_{i,d}\|.
	\end{align*}
	Thus, using Conditions~\ref{cond:distant} and~\ref{cond:orthormalBasis}, we have $\liminf_{d \to \infty} \|f_{j,d} - f_{i,d}\| > \Delta$.
\end{proof}

Henceforth, whenever both Conditions~\ref{cond:distant} and \ref{cond:orthormalBasis} are verified, we will denote by $d_0$ the integer given in Lemma \ref{lemma:existd} and which guarantees that its conclusions hold.  
\smallskip

In everything that follows, we will write $a \ \triangleleft \  b$ whenever there exists a positive constant $C$ which does not depend on $a$, or $b$, but may depend on some underlying fixed parameter, such that $a \le Cb$.
\begin{Condition}
	\label{cond:covFunctCondition}		
	The parametric set of covariance functions given by \eqref{eq:kernel} is supposed to be stationary, that is we can write the kernels $K_{\theta}$ as 
	\begin{equation*}
		\forall \theta \in \Theta,\forall f,g \in \functionSpace, \ K_\theta(f, g) = K_\theta(f - g).
	\end{equation*}
	We also suppose that both Conditions~\ref{cond:distant} and \ref{cond:orthormalBasis} are satisfied, and that there exists a fixed $\gamma_0 > d_0$, such that for all $f \in \functionSpace$ we have
	\begin{equation*}
		\sup_{\theta \in \Theta} |K_\theta(f)| \ \triangleleft \  \frac{1}{1+\|f\|^{\gamma_0}}.
	\end{equation*}
\end{Condition}

    Condition \ref{cond:covFunctCondition} requires the underlying GP to be stationary. Note that a special case of a stationary process is an isotropic process, where for every isometry $\tau$ of $\functionSpace$, it holds for any $f,g \in \functionSpace$
	\begin{equation}
        K_0(f,g) = K_0(\|f -  g\|)
        %= K_0(\|\tau(f) -  \tau(g)\|)
        = \cov(Y(f), Y(g))
        =
        \cov(Y(\tau(f)), Y(\tau(g))).
        %= K_0(f - g)
		\label{eq:stationarity}
	\end{equation}
	%Note that, if the covariance of $Y$ is a function of $\|f-g\|$, then~\eqref{eq:stationarity} is verified.
    An example of an isotropic covariance function is the SE kernel in~\eqref{eq:statkernel}.

    %Note from Condition~\ref{cond:covFunctCondition} that the isotropy of the field is not required. However, if we assume the field to be isotropic then certain conditions such as Condition \ref{cond:covFunctCondition}, Condition \ref{cond:dcovFunctCondition}, Condition \ref{cond:d3covFunctCondition} or Condition \ref{cond:dcovFunctCondition} can be simplified.

The inverse of the matrix $R_{\theta}$ also plays a key role in the definition of the ML estimator as can be seen in~\eqref{eq:negLogGaussianLike}. As such, any convergence result regarding this estimator will require some form of control over $R_{\theta}^{-1}$ to avoid some approximation-related divergence phenomenon whenever $|R_{\theta}|$ is close to zero. Thus, Condition~\ref{cond:infEingenVals} can be understood as a stability condition. 
\begin{Condition}
	\label{cond:infEingenVals}	
	There exists a fixed $c > 0$ such that for any $\theta \in \Theta$, the sequence of matrices $R_\theta = (K_\theta(f_i, f_j))_{1 \leq i,j \leq n}$ satisfies
	\begin{equation*}
		%\sup_{\theta \in \Theta}
		\lambda_{\min}(R_\theta) \geq c,
	\end{equation*}
	where $\lambda_{\min}(R_\theta)$ denotes the smallest eigenvalue of $R_\theta$. Thus $
	%\sup_{\theta \in \Theta} 
	\lambda_{\max} (R_{\theta}^{-1}) \leq 1/c$, where $\lambda_{\max}(R_\theta^{-1})$ denotes the largest eigenvalue of $R_\theta^{-1}$.
\end{Condition}
For instance, in the case of a stationary covariance function given by $K_{\theta}(\omega) = \overline{K}_{\theta}(w) + \delta_{\theta}\textbf{1}_{\{\omega = 0\}}$ where $\overline{K}_{\theta}$ is a continuous covariance function and $\inf_{\theta \in \Theta} \delta_{\theta} > 0$, then Condition \ref{cond:infEingenVals} is always satisfied.

A very common hypothesis within parametric estimation is that of the separability of the covariance parameters~\cite{Bachoc2014AsymptAnalGPs,Bachoc2018_GPdistInputs}. Condition~\ref{cond:diffK} imposes a form of global separability for the estimated parameter. Indeed, it can be interpreted as follows: given enough input data, it is always possible to distinguish between the true covariance parameter $\theta_0$ and any other parameter $\theta$. This condition is crucial to ensure that the studied problem is well-posed. %More precisely, if Condition~\ref{cond:diffK} is not satisfied, then it is not possible to guarantee that the true covariance parameter $\theta_0$ is distinguishable from any other covariance parameter.
\begin{Condition}
	\label{cond:diffK}	
	For all $\alpha >0$,
	\begin{equation*}
		\liminf_{n \to \infty} \inf_{\|\theta-\theta_0\| \geq \alpha} \frac{1}{n} \sum_{i,j = 1}^{n} [K_\theta(f_i,f_j)-K_{\theta_0}(f_i,f_j)]^2 > 0.
	\end{equation*}
\end{Condition}
Condition~\ref{cond:positivity} can be interpreted as the asymptotic linear independence of the derivatives of the covariance function at $\theta _0$. In particular, this condition will ensure that the asymptotic covariance matrix $\asympConMat$ given by \eqref{eq:asympConMat} is well-defined, i.e. that it is indeed a covariance matrix. 
\begin{Condition}
	\label{cond:positivity}
	For all $(\lambda_1, \ldots, \lambda_p) \ne (0, \ldots, 0),$
	\begin{equation*}
		\liminf_{n \to \infty} \frac{1}{n} \sum_{i,j = 1}^{n} \left(\sum_{k=1}^{p} \lambda_k \frac{\partial K_{\theta}}{\partial \theta_k} \bigg|_{\theta = \theta_0} \hspace{-3.5ex}(f_i,f_j) \right)^2 > 0.
	\end{equation*}
\end{Condition}
Both Conditions~\ref{cond:dcovFunctCondition} and \ref{cond:d3covFunctCondition} are standard regularity conditions within the ML estimation with Gaussian data. For these two conditions, we first need to assume that Conditions \ref{cond:distant} and~\ref{cond:orthormalBasis} are satisfied.
\begin{Condition}
	\label{cond:dcovFunctCondition}
	For any $f \in \functionSpace$, $K_\theta(f)$ is continuously differentiable with respect to $\theta$ and, for some $\gamma_1 > d_0$ with $d_0$ as in Lemma~\ref{lemma:existd}, we have
	\begin{equation*}
		\sup_{\theta \in \Theta} \max_{i = 1, \ldots, p} \left|\frac{\partial K_\theta}{\partial \theta_i}(f)\right| \ \triangleleft \ \frac{1}{1+\|f\|^{\gamma_1}}.
	\end{equation*}
\end{Condition}

\begin{Condition}
	\label{cond:d3covFunctCondition}	
	For any $f \in \functionSpace$, $K_\theta(f)$ is three times continuously differentiable with respect to $\theta$ and we have, for $q \in \{2,3\}$, $i_1, \ldots, i_q \in \{1, \ldots, p\}$, and for some $\gamma_2 > d_0$ with $d_0$ as in  Lemma~\ref{lemma:existd},
	\begin{equation*}
		\sup_{\theta \in \Theta} \left|\frac{\partial^q K_\theta}{\partial \theta_{i_1}\cdots \partial \theta_{i_q}} (f)\right| \ \triangleleft \ \frac{1}{1+\|f\|^{\gamma_2}}.
	\end{equation*}
\end{Condition}

\paragraph{The case of isotropic processes.}
All the conditions mentioned above hold for any stationary GP even if the process is not isotropic. However, when isotropy is indeed assumed, then we can relax some of the previous conditions to more tractable versions. The adaptations of those conditions to the isotropic case are summarized in Appendix~\ref{app:CondIsotropy}. %There, Condition~\ref{cond:covFunctConditionIso} is assumed to be verified for the GP $Y$.
\smallskip

We next give an example that covers the case of non-randomly perturbed grids in $\mathbb{N}^{d_0}$ with $d_0$ as in  Lemma~\ref{lemma:existd}. Note that small perturbations of regular grids are often considered to improve estimation by creating pairs of observations that are less independent
\cite{Bachoc2014AsymptAnalGPs,furrer2023asymptotic}.
\begin{example}
    \label{example:perturbed}
	Let $(e_n)_{n \in \mathbb{N}}$ be an orthonormal basis of $\functionSpace = L^2([0,1])$ and $(x_i)_{i \in \mathbb{N}} \subset \mathbb{N}^{d_0}$ be deterministic such that for all $m \in \mathbb{N}$,
	\begin{equation*}
		\{x_1, \ldots, x_{m^{d_0}}\} = \{1, \ldots, m\}^{d_0}.
	\end{equation*}	
 We consider $\delta < \frac{1}{2}$ fixed and we define 
 \begin{equation*}
		f_i = \sum_{k=1}^{d_0} ((x_i)_k + \delta \xi_{i,k}) e_k + \sum_{k = d_0+1}^{\infty} \xi_{i,k} s_k e_k,
	\end{equation*}
	for $i \in \mathbb{N}$, where $(\xi_{i,k})_{i, k \in \mathbb{N}}$ is deterministic in $[-1,1]$. In that case, Condition \ref{cond:distant} is satisfied with $\Delta \le \frac{{1-2 \delta}}{2}$.
 %Suppose that Condition \ref{cond:distant} is satisfied with $\Delta < \frac{1}{2}$ and consider $\delta$ fixed such that $1-2\delta \geq 2 \Delta$. For $i \in \mathbb{N}$, we define
Now, let $\theta = (\theta_1,\theta_2)$ with $\Theta$ compact in $(0, \infty)^2$ and $K_{\theta}(f,g) = F_{\theta} (\|f-g\|) = \theta_1 e^{-\theta_2 \|f-g\|}$ for all $f, g \in  L^2([0,1])$. Then, Conditions~\ref{cond:diffK}, and~\ref{cond:positivity} are also verified. 
	
	%\medspace
	%To satisfy the condition (ident global) is almost sure (photo 19/04/2023).
	%		Si $\forall \theta$, $F_\theta$ et $F_{\theta_0}$ ne sont pas p.s. égales (au sens mesure de Lebesgue) sur
	%		\begin{equation*}
		%			D_{n,\delta,\thr}:= \left\{
		%			\left\|
		%			\sum_{j=1}^{d_0} ((x_i)_j + \delta \xi_j) e_j 
		%			- \sum_{j=1}^{d_0} ((x_k)_j + \delta \xj_j) e_j
		%			+ \sum_{j = d_0+1}^\infty \xi_j s_j e_j
		%			- \sum_{j = d_0+1}^\infty \xi'_j s_j e_j
		%			\sum_{j=1}^{d_0+\thr} ((x_i)_j + \delta \xi_{i,j}) e_j 
		%			- \sum_{j=1}^{d_0+\thr} ((x_{i'})_j + \delta \xi_{i',j}) e_j
		%			+ \sum_{j = d_0+\thr+1}^\infty \xi_{i,j} s_j e_j 
		%			- \sum_{j = d_0+\thr+1}^\infty \xi_{i',j} s_j e_j
		%			\right\|;
		%			i, k \in \mathbb{N}, i \ne k, \xi, \xi' \in [-1,1]^\infty
		%			\right\}.
		%		\end{equation*}
	%		Alors, conditions \eqref{cond:diffK} et \eqref{cond:positivity} sont vraies %p.s.
\end{example}

\begin{proof}
	%	Suppose that there exists $\alpha > $ such that
	
	%	$d_0 = J_0$
	
	%	$\delta< \frac{1}{2}$
	
	%	$\xi_{i,j}, i = 1, \ldots, n, j \in \mathbb{N}$ iid loi compact $\subseteq [-1,1]$
	
    Let $n \in \mathbb{N}^*$ and $1 \le i \ne j \le n$. There exists $k \in \{1, \ldots, d_0\}$ such that $(x_i)_{k} \ne (x_{j})_{k}$. Since, $\delta < \frac{1}{2}$, we have
    \begin{equation}
    \label{eq:distance_lower_bound}
		\left\|
		\sum_{k=1}^{d_0} \left[(x_i - x_j)_k + \delta \xi_{i,j,k}^{-} \right] e_k
		+ \sum_{k = d_0+1}^\infty s_k e_k \xi_{i,j,k}^{-} 
		\right\| \geq {1-2\delta} \ge 2\Delta,  
    \end{equation}
    with $\xi_{i,j,k}^{-} = \xi_{i,k} - \xi_{j,k}$ and $\Delta$ as mentioned in Example~\ref{example:perturbed}. In other words, Condition \ref{cond:distant} is satisfied with  $\Delta \le \frac{{1-2\delta}}{2}$.
    \smallskip
    
    Before establishing that Condition~\ref{cond:diffK}, and~\ref{cond:positivity} are also verified under the hypothesis of Example~\ref{example:perturbed}, we show a useful upper bound on the distance between two inputs. Consider $\theta = (\theta_1,\theta_2) \in \Theta$ and $F_{\theta}$  as described in Example~\ref{example:perturbed}. Let $m$ be the largest integer such that $m^{d_0} \leq n$. Then $m^{d_0} / n$ converges to $1$ as $n \to \infty$ since $d_0$ is fixed. Then for all $i \in \{1 ,\ldots , m^{d_0}\}$, there is a corresponding $j$ such that
	\begin{align*}
		(x_i)_1 - (x_{j})_1 &= \pm 1,
		\\
		(x_i)_k - (x_{j})_k &= 0, \quad k = 2, \ldots, d_0,
	\end{align*}
	and thus 
	\begin{equation}
            \label{eq:distance_upper_bound}
		\left\|
		\sum_{k=1}^{d_0} \left[(x_i - x_j)_k + \delta \xi_{i,j,k}^{-} \right] e_k
		+ \sum_{k = d_0+1}^\infty s_k e_k \xi_{i,j,k}^{-}
		\right\| \leq 1 + 2\delta (\sqrt{d_0-1}+1) + 2 R_{d_0},  
	\end{equation}
    where $R_{d_0}^2 = \sum_{k=d_0+1}^{\infty}{s_k^2}$.
    \smallskip
    
    Now, we prove Conditions~\ref{cond:diffK} and~\ref{cond:positivity}. Denote $G_{\theta, \theta'} (t) = F_{\theta}(t) - F_{\theta'}(t)$ and
	\begin{equation*}
		D_{n,\delta} = 
		\frac{1}{n} \sum_{i,j= 1}^n
		G_{\theta, \theta'}^2\left(
		\left\|
		\sum_{k=1}^{d_0} \left[(x_i - x_j)_k + \delta \xi_{i,j,k}^{-} \right] e_k
		+ \sum_{k = d_0+1}^\infty s_k e_k \xi_{i,j,k}^{-}
		\right\|
		\right).
	\end{equation*}
 \noindent
{Suppose that $\theta_1 \ne \theta'_1$}, then
	\begin{equation*}
		D_{n, \delta} \geq \frac{1}{n} \sum_{i = 1}^n G_{\theta,\theta'}^2(0)\geq \frac{1}{n} \sum_{i = 1}^n (\theta_1 - \theta'_1)^2 = (\theta_1 - \theta_1')^2,
	\end{equation*}
	and thus $\liminf_{n \to \infty} D_{n,\delta} >0$. 
%	\begin{equation*}
%		|e^{-\theta_2 L} - e^{-\theta'_2 L}| \geq \min(\theta_2, \theta'_2) e^{-\max(\theta_2,\theta'_2) L } |\theta_2-\theta'_2|.
%	\end{equation*}
    Now, without loss of generality, assume that {$\theta_1 = \theta'_1 = 1$ and $\theta_2 \ne \theta'_2$}. According to the convexity of the function $\theta_2 \mapsto e^{-\theta _2 L}$, we have for any positive $L$,
    \begin{equation*}
        \left|e^{-\theta_2L} - e^{-\theta_2'L}\right| \ge L\left|\theta_2 - \theta_2'\right|e^{-\max(\theta_2,\theta_2')L}.
    \end{equation*}
	%\begin{equation*}
	%	D_{n, \delta} \geq \frac{1}{n} \sum_{i = 1}^{m^{d_0}} %\min(\theta_2, \theta'_2) |\theta_2 - \theta'_2| e^{- %\max(\theta_2, \theta'_2) [1 + \delta \sqrt{d_0} + 2 R_{d_0}]}. 
	%\end{equation*}
    Thus, using \eqref{eq:distance_lower_bound} and~\eqref{eq:distance_upper_bound}, we get
        \begin{equation*}
	D_{n, \delta} \geq \frac{(1-2\delta)^2}{n} \sum_{i = 1}^{m^{d_0}} (\theta_2 - \theta_2')^2e^{-2\max(\theta_2,\theta_2')[1 + 2\delta \left(\sqrt{d_0-1}+1\right) + 2 R_{d_0}]}. 
	\end{equation*}
	We see that $\liminf_{n \to \infty } {D}_{n, \delta} >0$.  This concludes the proof for Condition~\ref{cond:diffK}.
    \smallskip
    
	To show that Condition~\ref{cond:positivity} holds, we let $(\lambda_1,\lambda_2) \not= (0,0)$, $\theta_0 = (\theta_1, \theta_2)$. Note that $\frac{\partial F_\theta(t)}{\partial \theta_1}\Big|_{\theta = \theta_0} \hspace{-1ex} = e^{-\theta_2 t}$ as well as $\frac{\partial F_\theta(t)}{\partial \theta_2}\Big|_{\theta = \theta_0} \hspace{-1ex} = -\theta_1 t e^{-\theta_2 t}$ for all $t \geq 0$. Then,
	\begin{align*}
		\liminf_{n \to \infty} & \frac{1}{n} \sum_{i,j = 1}^{n} \left[\sum_{k=1}^{p} \lambda_k \left.\frac{\partial K_{\theta}}{\partial \theta_k} \right|_{\theta = \theta_0} \hspace{-4ex}(f_i,f_j) \right]^2 
		= \liminf_{n \to \infty} \frac{1}{n} \sum_{i,j = 1}^{n} \left(
		[\lambda_1 - \lambda_2 \theta_1 g_{ij}] e^{-\theta_2 g_{ij}}
		\right)^2,
	\end{align*}
        where $g_{ij} = \|f_i - f_j\|$. For $\lambda_1 \ne 0$,
	\begin{align*}
		\liminf_{n \to \infty} \frac{1}{n} \sum_{i,j = 1}^{n} \left(
		[\lambda_1 - \lambda_2 \theta_1 g_{ij}] e^{-\theta_2 g_{ij}}
		\right)^2
            \geq
            \liminf_{n \to \infty} \frac{1}{n} \sum_{i = 1}^{n} \lambda_1^2
		> 0.
	\end{align*}
	If $\lambda_1 = 0$ and $\lambda_2 \ne 0$, then using \eqref{eq:distance_upper_bound}, we obtain
	\begin{align*}
		\liminf_{n \to \infty} \frac{1}{n} \sum_{i,j = 1}^{n} 
		\lambda_2^2 \theta_1^2 g_{ij}^2 e^{-2\theta_2 g_{ij}}
		\geq 
            \liminf_{n \to \infty} \frac{4}{n} \sum_{i = 1}^{m^{d_0}} \lambda_2^2 \theta_1^2 \Delta^2  e^{- 2 \theta_2 [1 + 2\delta (\sqrt{d_0-1}+1) + 2 R_{d_0}]}
		> 0.
	\end{align*}
    This concludes the proof of Condition~\ref{cond:positivity}.
\end{proof}

\section{Asymptotic results}
\label{sec:asympProperties}

In this section, we introduce the theorems related to the asymptotic properties of the ML estimator of a GP with functional inputs. For conciseness, the proofs of all the technical lemmas are summarized in Appendix~\ref{app:proofLemmas}.

%In this section, we give the asymptotic results obtained within the framework mentioned in the previous section.

\subsection{Asymptotic consistency}
\label{sec:asympProperties:subsec:consistency}

We first focus on the asymptotic consistency of the ML estimator. Before presenting the main theorem, it is worth establishing the following technical lemmas to provide a detailed proof.

 The following lemma, when combined with Conditions~\ref{cond:distant} and~\ref{cond:orthormalBasis}, enables control of the $L^1$-norm of the covariance matrix $R_\theta$. This, coupled with an application of Gershgorin's circle theorem~\cite{Li2018GersgorinTheo}, allows for the control of the eigenvalues of $R_\theta$.
\begin{lemma}
    \label{lemma:A}
    For $g:\mathbb{R}^+ \to \mathbb{R}$ a non-increasing function and $(x_i)_{i \in \mathbb{N}^*}$ a sequence of real non-negative numbers, we have for any $n \in \mathbb{N}^*$,
    \begin{equation*}
        \sum_{i=1}^{n} g(x_i)
	\le \sum_{k=1}^{\infty} \nu_k g({k-1}),
    \end{equation*}
     with $\nu_k = \#\{i \ : \ x_i \in [k-1,k)\}$, $k \in \mathbb{N}^*$.
\end{lemma}

\begin{lemma}
	\label{lemma:S}
	Assume that Conditions \ref{cond:distant} to \ref{cond:covFunctCondition} hold. Then
	\[
	\sup_{\theta \in \Theta} \max_{j=1,\ldots,n} \sum_{i=1}^{n} |K_\theta(f_j, f_i)|
	\]
	is bounded as $n \to \infty$.
\end{lemma}

\begin{lemma}
\label{lemma:S'}
Suppose that Conditions \ref{cond:distant} to \ref{cond:covFunctCondition}, and Condition \ref{cond:dcovFunctCondition} are all satisfied. Then, for $k = 1, \ldots, p$,
	\[
	\sup_{\theta \in \Theta} \max_{j=1,\ldots,n} \sum_{i=1}^{n} \left|\frac{\partial K_\theta}{\partial \theta_k} (f_j, f_i)\right|
	\]
	is bounded as $n \to \infty$.
\end{lemma}
Note that only the first part of Condition \ref{cond:covFunctCondition}, i.e. the stationary requirement, is needed for Lemma~\ref{lemma:S'} to hold. However, we have decided to unify both parts as one condition for clarity purposes thereafter.
\begin{lemma}
    \label{lemma:lambdaMax}
    Under the same assumptions as in Lemma \ref{lemma:S'}, we have that
	\begin{equation*}
		\sup_{\theta \in \Theta} \lambda_{\max} (R_\theta) \qquad \textrm{and} \qquad \sup_{\theta \in \Theta} \max_{i = 1, \ldots, p} \lambda_{\max} \left(\frac{\partial R}{\partial \theta_i}(\theta) \right)
	\end{equation*}
	are bounded as $n \to \infty$.
\end{lemma}

Lemma \ref{lemma:consitencyLike} requires us to introduce a new item of notation. For any sequence of random variables $(X_n)_{n \in \mathbb{N}}$, we use the probabilistic Landau notation, that is
\[
X_n = O_{\mathbb{P}}(1) \qquad \textrm{means} \qquad  \lim_{M \to \infty} \limsup_{n \to \infty} \, \mathbb{P}(|X_n| > M) = 0.
\]
\begin{lemma}
    \label{lemma:consitencyLike}
	Suppose that $M:=\max_{i=1,\ldots,p}\sup_{\theta \in \Theta}\left|\frac{\partial \negloglikeFun_{n}}{\partial \theta_{i}}(\theta)\right| = O_{\mathbb{P}}(1)$ and that $\negloglikeFun_{n}(\theta) - \mathbb{E}\left[\negloglikeFun_{n}(\theta)\right] \xrightarrow[n \to \infty]{\mathbb{P}} 0$ for all $\theta \in \Theta$, then
	$$
	\sup_{\theta \in \Theta}\left|\negloglikeFun_{n}(\theta) - \mathbb{E}\left[\negloglikeFun_{n}(\theta)\right]\right|\xrightarrow[n \to \infty]{\mathbb{P}} 0.
	$$
\end{lemma}
We recall that the proofs of the lemmas outlined in this section are summarized in Appendix~\ref{app:proofLemmas}.
\smallskip

Having introduced the aforementioned lemmas, we can now establish Theorem~\ref{theo:consistency} which gives the asymptotic consistency of the ML estimator.
\begin{theorem}
    \label{theo:consistency}
    Let $\thetaML \in \Theta$ be the maximum likelihood estimator obtained by minimizing the modified negative log Gaussian likelihood in~\eqref{eq:negLogGaussianLike}. Under Conditions~\ref{cond:distant} to~\ref{cond:diffK}, as well as Condition~\ref{cond:dcovFunctCondition}, we have
    \begin{equation*}
	\thetaML \xrightarrow[n \to \infty]{\mathbb{P}} \theta_0.
    \end{equation*}
\end{theorem}
%This proof follows the steps in the proof of Theorem V.13 in \cite{Bachoc2018_GPdistInputs}.

\begin{proof}[Proof of Theorem \ref{theo:consistency}]
    We start the proof in the same manner as in the proof of Proposition 3.1 in \cite{Bachoc2014AsymptAnalGPs}. Indeed, we first aim to obtain the following uniform convergence
	\begin{equation}
		\sup_{\theta \in \Theta} \left|\negloglikeFun_{n}(\theta) - \expect(\negloglikeFun_{n}(\theta))\right| \xrightarrow[n \to \infty]{\mathbb{P}} 0.
		\label{eq:op1}
	\end{equation}
    Let $n \in \mathbb{N}^*$ and recall that $\thetaML \in \argmin_{\theta \in \Theta} \negloglikeFun_{n}(\theta)$ with
	\begin{equation*}
		\negloglikeFun_{n}(\theta) = \frac{1}{n} \ln|R_\theta| + \frac{1}{n} y^\top R_\theta^{-1} y.
	\end{equation*}
    Let $\theta \in \Theta$, it holds
    \[
     \Var\left(\negloglikeFun_{n}(\theta)\right)= \frac{2}{n^2}\tr\left(R_{\theta}^{-1}R_{\theta_0}R_{\theta}^{-1}R_{\theta_0}\right). 
    \]
    However, according to Lemma~\ref{lemma:lambdaMax} the quantity $\sup_{\theta \in \Theta} \lambda_{\max} (R_\theta)$ is bounded in $n$. Therefore, we have the convergence $\Var\left(\negloglikeFun_{n}(\theta)\right) \xrightarrow[n \to \infty]{} 0$ and so $\negloglikeFun_{n}(\theta) - \mathbb{E}\left(\negloglikeFun_{n}(\theta)\right) = o_{\mathbb{P}}(1)$.
    Moreover,
    \begin{equation*}
		\sup_{\theta \in \Theta} \max_{i = 1, \ldots, p} \lambda_{\max} \left(\frac{\partial R}{\partial \theta_i} (\theta)\right)
	\end{equation*}
    is also bounded in $n$. Thus, applying Lemma \ref{lemma:consitencyLike}, we also have the uniform convergence in $\theta$ as given by \eqref{eq:op1}.
    \smallskip
    
    In a second time, we shall establish that \begin{equation}
 \label{eq:lower_bound}
		\expect(\negloglikeFun_{n}(\theta) - \negloglikeFun_{n}(\theta_0)) \triangleright \left\|R_\theta - R_{\theta_0}\right\|_2^2,
	\end{equation}
 where $\|R\|_2^2 := \frac{1}{n} \sum_{i,j=1}^{n} R_{i,j}^2$. To do so, we notice that since $\sup_{\theta}\lambda_{\max}(R_{\theta})$ (resp. $\inf_{\theta}\lambda_{\min}(R_{\theta})$) is uniformly bounded (resp. uniformly bounded away from $0$) in $n$, then there exist $a,b > 0$ such that for all $n$ and for all $\theta \in \Theta$, the spectrum of $R_{\theta}^{-1}R_{\theta_0}$ lies within $[a,b]$. Now, let $\theta \in \Theta$ and notice that
    \[
    \mathbb{E}\left(\negloglikeFun_{n}(\theta)\right) = \frac{1}{n}\ln |R_{\theta}|+\frac{1}{n}\tr\left(R_{\theta}^{-1}R_{\theta_0}\right).
    \] 
    Denoting by $\lambda_1, \ldots, \lambda _n$ the eigenvalues of $R_{\theta}^{-1}R_{\theta_0}$, we get 
    \begin{align*}
        \expect(\negloglikeFun_{n}(\theta) - \negloglikeFun_{n}(\theta_0)) & = \frac{1}{n}\left(\tr\left(R_{\theta}^{-1}R_{\theta_0}\right) - \ln |R_{\theta}^{-1}R_{\theta_0}| - n \right)\\
        & =\frac{1}{n}\sum_{i=1}^n\left(\lambda _i - \ln \lambda_ i -1\right).
    \end{align*}
    Since, for all $t \in [a, b]$ we have $t - \ln t -1 \triangleright (t-1)^2$, then
    \begin{align*}
    \frac{1}{n}\sum_{i=1}^n\left(\lambda _i - \ln \lambda_ i -1\right) & \triangleright \frac{1}{n}\sum_{i=1}^n\left(\lambda _i -1 \right)^2\\
    & = \left\|I_n - R_{\theta}^{-1}R_{\theta_0}\right\|_2^2\\
    & = \left\|R_{\theta}^{-1}\left(R_{\theta} - R_{\theta_0}\right)\right\|_2^2.
    \end{align*}
    However since $\sup_{\theta \in \Theta}\lambda_{\max}(R_{\theta})$ is bounded in $n$, we have
    $$
    \left\|R_{\theta}^{-1}\left(R_{\theta} - R_{\theta_0}\right)\right\|_2^2 \triangleright \left\|R_{\theta} - R_{\theta_0}\right\|_2^2,
    $$
    and thus, we can conclude that \eqref{eq:lower_bound} is verified. From~\eqref{eq:op1} and Condition~\ref{cond:diffK}, we have the convergence $\mathbb{P}(\|\thetaML - \theta_0\| > \epsilon) \xrightarrow[n \to \infty]{} 0$, for all $\epsilon > 0$, and so the conclusion Theorem~\ref{theo:consistency} holds.
\end{proof}

\subsection{Asymptotic normality}
\label{sec:asympProperties:subsec:normality}

Similarly to Section~\ref{sec:asympProperties:subsec:consistency}, we first provide two technical lemmas that are necessary to establish the main theorem concerning asymptotic normality. We start by obtaining an extension of Lemma~\ref{lemma:S'} to the partial derivatives of orders two and three.
\begin{lemma}
	\label{lemma:S''}
	Suppose that Conditions \ref{cond:distant} to \ref{cond:covFunctCondition}, and \ref{cond:d3covFunctCondition} are satisfied, then for all $q \in \{2,3\}$ and for all $i_1, \ldots, i_q \in \{1, \ldots, p\}$, 
	\[
	\sup_{\theta \in \Theta} \max_{j=1,\ldots,n} \sum_{i=1}^{n} \left|\frac{\partial^q K_\theta}{\partial \theta_{i_1} \cdots \ \partial \theta_{i_q}} (f_j, f_i)\right|
	\]
	is bounded as $n \to \infty$.
\end{lemma}

Using Lemma~\ref{lemma:S''}, we can extend Lemma~\ref{lemma:lambdaMax} to the partial derivatives of orders two and three.
\begin{lemma}
	\label{lemma:lambdaMaxDeriatives}
        Assume that the same conditions as in Lemma \ref{lemma:S''} hold. Then, for all $q \in \{2,3\}$ and for all $i_1, \ldots, i_q \in \{1, \ldots, p\}$,
	\begin{equation*}
		\sup_{\theta \in \Theta} \lambda_{\max} \left(\frac{\partial^q R}{\partial \theta_{i_1}\cdots \ \partial \theta_{i_q}}(\theta) \right)
	\end{equation*}
	is bounded as $n \to \infty$.
\end{lemma}

%	\begin{Condition}
	%		\label{cond:dcovFunctConditiont}
	%		$\forall \theta \in \Theta$, $\forall t \geq 0$, $F_\theta(t)$ is continuously differentiable with respect to $t$ and, for $\ctedKdt < \infty$ and $\gamma  > 0$, we have
	%		\begin{equation*}
		%			\left|\frac{\partial}{\partial t}F_\theta(t)\right| \leq \ctedKdt e^{-\gamma t}.
		%		\end{equation*}
	%	\end{Condition}

We can now introduce Theorem~\ref{theo:normality} which yields the asymptotic normality of the ML estimator. 
%	\begin{theorem}
	%		Asymptotic normality. 
	%	\end{theorem}
\begin{theorem}
	\label{theo:normality}
	For all $n \in \mathbb{N}$, let $\asympConMat$ be the $p \times p$ asymptotic covariance matrix defined by
	\begin{equation}
	    \label{eq:asympConMat}
		(\asympConMat)_{i,j} = \frac{1}{2n} \tr\bigg(R_{\theta_0}^{-1} \frac{\partial R}{\partial \theta_i}(\theta_0) R_{\theta_0}^{-1} \frac{\partial R}{\partial \theta_j}(\theta_0)\bigg).
	\end{equation}
	Suppose that Conditions \ref{cond:distant} to \ref{cond:d3covFunctCondition} are all satisfied, then the maximum likelihood estimator $\thetaML$ is asymptotically normal:
	\begin{equation}
	\label{eq:AsympNorm}
		\sqrt{n} \asympConMat^{1/2} (\thetaML - \theta_0) \xrightarrow[n \to \infty]{\negloglikeFun} \mathcal{N}(0, I_p). 
	\end{equation}
	Furthermore,
	\begin{equation}
		0 < \liminf_{n \to \infty} \lambda_{\min}(\asympConMat) \leq \limsup_{n \to \infty} \lambda_{\max}(\asympConMat) < \infty.
		\label{eq:posEigVals}
	\end{equation}
\end{theorem}

%We are now ready to prove Theorem \ref{theo:normality}. We proceed as in (??).
\begin{proof}[Proof of Theorem \ref{theo:normality}]
	From Lemma~\ref{lemma:lambdaMax} and Condition~\ref{cond:infEingenVals}, we have for every $n \in \mathbb{N}$, $|(M_n)_{i,j}| \leq B$ for a fixed $B < \infty$. In addition, for any $\lambda_1, \ldots, \lambda_p \in \mathbb{R}$ such that $\sum_{i=1}^p \lambda_i^2 = 1$; from Condition~\ref{cond:infEingenVals}, we get that
	\begin{align*}
		\sum_{i,j = 1}^p \lambda_i \lambda_j (\asympConMat)_{i,j}
		&= \sum_{i,j = 1}^p \lambda_i \lambda_j \frac{1}{2n} \tr\bigg(R_{\theta_0}^{-1} \frac{\partial R}{\partial \theta_i}(\theta_0) R_{\theta_0}^{-1} \frac{\partial R}{\partial \theta_j}(\theta_0)\bigg)
		\\
		&= \frac{1}{2n} \tr\bigg(\sum_{i,j = 1}^p \lambda_i R_{\theta_0}^{-1} \frac{\partial R}{\partial \theta_i}(\theta_0) R_{\theta_0}^{-1} \lambda_j \frac{\partial R}{\partial \theta_j}(\theta_0)\bigg)
		\\
		&= \frac{1}{2n} \tr\bigg(R_{\theta_0}^{-1}  \bigg[\sum_{i = 1}^p \lambda_i \frac{\partial R}{\partial \theta_i}(\theta_0)\bigg] R_{\theta_0}^{-1} \bigg[\sum_{j = 1}^p \lambda_j \frac{\partial R}{\partial \theta_j}(\theta_0)\bigg] \bigg)
		\\
%		&= \frac{1}{2} \left(R_{\theta_0}^{-1/2}  \bigg[\sum_{i = 1}^p \lambda_i \frac{\partial R}{\partial \theta_i}(\theta_0)\bigg] R_{\theta_0}^{-1/2} \right)^2
%		\\
		&\triangleright\left\|\sum_{i = 1}^p \lambda_i \frac{\partial R}{\partial \theta_i}(\theta_0) \right\|_2^2,
	\end{align*}
    where we recall that $\|R\|_2^2 := \frac{1}{n} \sum_{i,j=1}^{n} R_{i,j}^2$.
%	with a fixed $C > 0$, since for every $n$
%	\begin{equation*}
%		\lambda_\min(R_{\theta_0}^{-1})
%		= \frac{1}{\lambda_\max(R_{\theta_0})}
%		\geq C
%		>0,
%	\end{equation*}
%	from Lemma~\ref{lemma:lambdaMax}. 
Hence, from Condition~\ref{cond:positivity}, we obtain
	\begin{equation*}
		\liminf_{n \to \infty} \lambda_{\min}(\asympConMat) > 0.
	\end{equation*}
	This establishes \eqref{eq:posEigVals}.
    \smallskip
    
 We move on to the proof of the asymptotic normality of the ML estimator. We start by assuming that
	\begin{equation}
		\sqrt{n} \asympConMat^{1/2} (\thetaML - \theta_0) \ \cancel{\xrightarrow[n \to \infty]{\negloglikeFun}} \ \mathcal{N}(0, I_p).
		\label{eq:negAsympNormality}
	\end{equation}
	Then there exists a bounded measurable function $g : \mathbb{R}^p \to \mathbb{R}$, $\xi > 0$, and a subsequence $\varphi(n)$ such that along $\varphi(n)$ we have
	\begin{equation*}
		\left|\expect\left(g\left(\sqrt{\varphi(n)} M_{\varphi(n)}^{1/2} (\widehat{\theta}_{\varphi(n)} - \theta_0) \right)\right) - \expect(g(U))\right| \geq \xi,
	\end{equation*}
	with $U \sim \mathcal{N}(0, I_p)$. In addition, by compactness and up to extracting another subsequence, we can assume that
	\begin{equation*}
		M_{\varphi(n)} \xrightarrow[n \to \infty]{\normalfont{a.s.}} M_\infty,
	\end{equation*}
	where $M_\infty$ is a symmetric positive definite matrix.
	
	Now the remaining of the proof is similar to the proof of Proposition 3.2 in~\cite{Bachoc2014AsymptAnalGPs}. Recall that for all $n$ and all $\theta$, we have
    \begin{equation*}
		\negloglikeFun_{n}(\theta) = \frac{1}{n} \ln|R_{\theta}| + \frac{1}{n} y^\top R_\theta^{-1} y,
	\end{equation*}
    then, for any $i \in \{1, \ldots, p\}$,
	\begin{align*}
		\frac{\partial \negloglikeFun_{n}}{\partial \theta_i}(\theta_0) 
		&= \frac{1}{n} \tr\bigg(R_{\theta_0}^{-1} \frac{\partial R}{\partial \theta_i}(\theta_0) \bigg) - \frac{1}{n}y^\top R_{\theta_0}^{-1} \frac{\partial R}{\partial \theta_i}(\theta_0) R_{\theta_0}^{-1} y.
	\end{align*}
	Hence, as in Proposition Appendix D.9 in~\cite{Bachoc2014AsymptAnalGPs}, we get that
	\begin{equation*}
		\sqrt{n} \nabla \negloglikeFun_{n} (\theta_0) \xrightarrow[n \to \infty]{\negloglikeFun} \mathcal{N}(0, 4 M_\infty).
	\end{equation*}
    For the sake of compactness and readability, we denote $\frac{\partial R_{\theta_0}}{\partial \theta_i} := \frac{\partial R}{\partial \theta_i}(\theta_0)$ and $\frac{\partial^2 R_{\theta_0}}{\partial \theta_i \partial \theta_j} = \frac{\partial^2 R}{\partial \theta_i \partial \theta_j} (\theta_0)$. Let $i,j \in \{1, \ldots, p\}$, then we have
	\begin{align*}
		\frac{\partial^2 \negloglikeFun_{n}}{\partial \theta_i \partial \theta_j}(\theta_0)
		=& \frac{1}{n} \tr\bigg(-R_{\theta_0}^{-1} \frac{\partial R_{\theta_0}}{\partial \theta_i} R_{\theta_0}^{-1} \frac{\partial R_{\theta_0}}{\partial \theta_j} + R_{\theta_0}^{-1} \frac{\partial^2 R_{\theta_0}}{\partial \theta_i \partial \theta_j} \bigg)
		%\\ 
		%& + \frac{1}{n} y^\top \bigg(R_{\theta_0}^{-1} \frac{\partial R}{\partial \theta_i}(\theta_0) R_{\theta_0}^{-1} \frac{\partial R}{\partial \theta_j}(\theta_0) R_{\theta_0}^{-1}
        %\\ 
        %&\hspace{9ex} + R_{\theta_0}^{-1} \frac{\partial R}{\partial \theta_j}(\theta_0) R_{\theta_0}^{-1} \frac{\partial R}{\partial \theta_i}(\theta_0) R_{\theta_0}^{-1}
        %-  R_{\theta_0}^{-1} \frac{\partial^2 R}{\partial \theta_i \partial \theta_j}(\theta_0) R_{\theta_0}^{-1} \bigg) y.
        %\\
        %& + \frac{1}{n} y^\top R_{\theta_0}^{-1} \bigg( \frac{\partial R}{\partial \theta_i}(\theta_0) R_{\theta_0}^{-1} \frac{\partial R}{\partial \theta_j}(\theta_0) + \frac{\partial R}{\partial \theta_j}(\theta_0) R_{\theta_0}^{-1} \frac{\partial R}{\partial \theta_i}(\theta_0)  
        %-  \frac{\partial^2 R}{\partial \theta_i \partial \theta_j}(\theta_0) \bigg) R_{\theta_0}^{-1} y.    
        \\
        &+ \frac{1}{n} y^\top R_{\theta_0}^{-1} \bigg( \frac{\partial R_{\theta_0}}{\partial \theta_i} R_{\theta_0}^{-1} \frac{\partial R_{\theta_0}}{\partial \theta_j} + \frac{\partial R_{\theta_0}}{\partial \theta_j} R_{\theta_0}^{-1} \frac{\partial R_{\theta_0}}{\partial \theta_i}
        -  \frac{\partial^2 R_{\theta_0}}{\partial \theta_i \partial \theta_j} \bigg) R_{\theta_0}^{-1} y.
	\end{align*}
        Thus, it follows from Condition~\ref{cond:infEingenVals}, Lemma \ref{lemma:lambdaMax} and Lemma~\ref{lemma:lambdaMaxDeriatives}, that 
        \begin{equation}
        \label{eq:conv_partial_2_var}
        \Var\left(\frac{\partial^2 \negloglikeFun_{n}}{\partial \theta_i \partial \theta_j}(\theta_0)\right) \xrightarrow[n \to \infty]{} 0.
        \end{equation}
        Moreover, we also have
	\begin{align*}
		\expect\left(\frac{\partial^2 \negloglikeFun_{n}}{\partial \theta_i \partial \theta_j}(\theta_0) \right)
		= & \frac{1}{n} \tr\left(-R_{\theta_0}^{-1} \frac{\partial R_{\theta_0}}{\partial \theta_i} R_{\theta_0}^{-1} \frac{\partial R_{\theta_0}}{\partial \theta_j} + R_{\theta_0}^{-1} \frac{\partial^2 R_{\theta_0}}{\partial \theta_i \partial \theta_j} \right)
		%\\
		%& + \frac{1}{n} \tr\left(\bigg(2 R_{\theta_0}^{-1} \frac{\partial R_{\theta_0}}{\partial \theta_i} R_{\theta_0}^{-1} \frac{\partial R_{\theta_0}}{\partial \theta_j} R_{\theta_0}^{-1} -  R_{\theta_0}^{-1} \frac{\partial^2 R_{\theta_0}}{\partial \theta_i \partial \theta_j} R_{\theta_0}^{-1} \bigg) R_{\theta_0}\right)
		\\
        & + \frac{1}{n} \tr\bigg(2 R_{\theta_0}^{-1} \frac{\partial R_{\theta_0}}{\partial \theta_i} R_{\theta_0}^{-1} \frac{\partial R_{\theta_0}}{\partial \theta_j} -  R_{\theta_0}^{-1} \frac{\partial^2 R_{\theta_0}}{\partial \theta_i \partial \theta_j} \bigg)
		\\
		%			&= \frac{1}{n} \tr\left(-R_{\theta_0}^{-1} \frac{\partial R_{\theta_0}}{\partial \theta_i} R_{\theta_0}^{-1} \frac{\partial R_{\theta_0}}{\partial \theta_j} + R_{\theta_0}^{-1} \frac{\partial^2 R_{\theta_0}}{\partial \theta_i \partial \theta_j} + 2 R_{\theta_0}^{-1} \frac{\partial R_{\theta_0}}{\partial \theta_i} R_{\theta_0}^{-1} \frac{\partial R_{\theta_0}}{\partial \theta_j} -  R_{\theta_0}^{-1} \frac{\partial^2 R_{\theta_0}}{\partial \theta_i \partial \theta_j} \right)
		%			\\
		= & \frac{1}{n} \tr\left(R_{\theta_0}^{-1} \frac{\partial R_{\theta_0}}{\partial \theta_i} R_{\theta_0}^{-1} \frac{\partial R_{\theta_0}}{\partial \theta_j} \right)
		\\
		= & \left(2 \asympConMat\right)_{i,j}.
	\end{align*}
	This yields, using \eqref{eq:conv_partial_2_var}, the following convergence in probability 
 \begin{equation}
    \label{eq:conv_partial_2}
     \frac{\partial^2 \negloglikeFun_{n}}{\partial \theta_i \partial \theta_j}(\theta_0) \xrightarrow[n \to \infty]{\mathbb{P}} (2 M_\infty)_{i,j}.
 \end{equation}
	
	Computing the third-order derivatives of the modified negative log-likelihood, we obtain
	\begin{align*}
		\frac{\partial^3 \negloglikeFun_{n}}{\partial \theta_i \partial \theta_j \partial \theta_k}(\theta_0)
		=& \frac{1}{n} \tr\left( \frac{\partial}{\partial \theta_k} \left(-R_{\theta}^{-1} \frac{\partial R_{\theta}}{\partial \theta_i} R_{\theta}^{-1} \frac{\partial R_{\theta}}{\partial \theta_j} + R_{\theta}^{-1} \frac{\partial^2 R_{\theta}}{\partial \theta_i \partial \theta_j} \right)\right)\bigg|_{\theta = \theta_0} 
		\\
		&+ \frac{1}{n} y^\top \bigg(\frac{\partial}{\partial \theta_k} \bigg[R_{\theta}^{-1} \frac{\partial R_{\theta}}{\partial \theta_i} R_{\theta}^{-1} \frac{\partial R_{\theta}}{\partial \theta_j}R_{\theta}^{-1}  + R_{\theta}^{-1} \frac{\partial R_{\theta}}{\partial \theta_j} R_{\theta}^{-1} \frac{\partial R_{\theta}}{\partial \theta_i} R_{\theta}^{-1} \\
        & \hspace{15ex} -   R_{\theta}^{-1} \frac{\partial^2 R_{\theta}}{\partial \theta_i \partial \theta_j} R_{\theta}^{-1} \bigg]_{\theta = \theta_0} \bigg) y
		\\
		= & \frac{1}{n} \tr(A_{\theta_0}) + \frac{1}{n} y^\top B_{\theta_0} y,
	\end{align*}
	where $A_{\theta_0}$ and $B_{\theta_0}$ are sums of products of the matrices $R_{\theta_0}^{-1}$ or $\frac{\partial R}{\partial\theta_{i_1}\cdots \partial\theta_{i_q}} (\theta_0)$ with $q \in \{0, \ldots, 3\}$ and  $i_1, \ldots, i_q \in \{1, \ldots, p\}$. Hence, from Condition \ref{cond:infEingenVals} and from Lemmas~\ref{lemma:lambdaMax}, and~\ref{lemma:lambdaMaxDeriatives}, we have
	\begin{equation*}
		\sup_{\theta \in \Theta} \left| \frac{\partial^3 \negloglikeFun_{n}}{\partial \theta_i \partial \theta_j \partial \theta_k}(\theta) \right| = O_{\mathbb{P}}(1).
	\end{equation*}
	Using \eqref{eq:conv_partial_2} and applying Proposition D.10 in~\cite{Bachoc2014AsymptAnalGPs}, we can show that
	\begin{equation*}
		\sqrt{\varphi(n)} (\widehat{\theta}_{\varphi(n)} - \theta_0) \xrightarrow[n \to \infty]{\negloglikeFun} \mathcal{N}(0, M_\infty^{-1}).
	\end{equation*}
	Since $M_{\varphi(n)} \xrightarrow[n \to \infty]{\normalfont{a.s.}} M_{\infty}$, we have
	\begin{equation*}
		\sqrt{\varphi(n)} M_{\varphi(n)}^{1/2}(\widehat{\theta}_{\varphi(n)} - \theta_0) \xrightarrow[n \to \infty]{\negloglikeFun} \mathcal{N}(0, I_p),
	\end{equation*}
	This is a contradiction with \eqref{eq:negAsympNormality}; which concludes the proof.
\end{proof}

%\clearpage

%\begin{lemma}
	%	Example of a perturbed grind in infinite dimensions.
	%\end{lemma}

\section{Asymptotic results with noisy inputs}
\label{sec:asympPropertiesApprox}

%\subsection{Motivation}
%\label{sec:asympPropertiesApprox:subsec:motivation}
In practical scenarios, complete knowledge of the functional inputs is often unattainable since only an approximate version of the input data is accessible. This approximation can be acquired through conventional sampling methods or through more advanced techniques such as projecting the input functions onto a fixed functional basis~\cite{Ramsay2005functional}. It is noteworthy that, when presenting numerical examples of certain asymptotic results for functional data, authors often tend to implicitly approximate their functional data using a sufficiently large sample vector. However, solely considering non-negligible approximation errors may render asymptotic results impractical. Thus, it is worthwhile to adapt the theorems in Section~\ref{sec:asympProperties} to ensure that the numerical approximation of the inputs does not affect the convergence results. 
\\
\\
To the best of our knowledge, this paper contains the first examination of the convergence of the ML estimator for approximated functional inputs within the GP framework. On a related topic for vector inputs,~\cite{furrer2023asymptotic} examines the impact of covariance approximations. This more in-depth analysis of the asymptotic results for the approximated data has revealed, in certain cases, some non-trivial conditions for the convergence of the estimator to occur.
\\
\\
The proofs of intermediate lemmas are outlined in Appendix~\ref{app:proofLemmasApprox}

\subsection{Notation and sufficient conditions}
\label{sec:asympPropertiesApprox:subsec:conditions}
Throughout this section, we use the notation introduced in Section~\ref{sec:asympProperties}. Given our focus on the asymptotic properties of approximated data, we consider some approximations of the input data $(f_i)_{i = 1,\ldots, n}$. For any $i \in \mathbb{N}$ and any $N \in \mathbb{N}$, we denote by $f_i^N$ the approximated version of $f_i$.
Similar to the case of noise-free inputs, our interest lies in studying the modified negative log Gaussian likelihood associated with the approximated data:
\begin{equation}
    \negloglikeFun_{n,N}(\theta) = \frac{1}{n} \ln|R_{N, \theta}| + \frac{1}{n} y^\top R_{N, \theta}^{-1} y,
    \label{eq:negLogGaussianLikeApprox}
\end{equation}
with $R_N(\theta) = R_{N, \theta} := (K_\theta(f_i^N, f_j^N))_{1 \leq i,j \leq n}$ and $|R_{N, \theta}|$ the determinant of $R_{N, \theta}$. Then, the ML estimator is given by $\thetaMLapprox \in \argmin_{\theta \in \Theta} \negloglikeFun_{n, N}(\theta)$. 
\medskip

For the asymptotic results to hold for approximated data, we will require that the approximation order $N$ satisfies Condition \ref{cond:approx:inputs}.
\begin{Condition} \label{cond:approx:inputs}
	There exists a sequence $N_n$ such that as $n \to \infty$, we have
	\[
	\max_{i = 1, \ldots, n}
	\| 
	f_i - f^{N_n}_i
	\|
	\xrightarrow[n \to \infty]{} 0.
	\]
\end{Condition}
In particular, it should be noted that this condition implies that the order of approximation $N$ must depend on the number of observations $n$ for convergence to hold in a general setting.

Additionally, for  Condition \ref{cond:covFunctCondition} to be correctly passed down to the approximated data, we will require some local control over the rate of change of the covariance function. This control is detailed in Condition~\ref{cond:derivative:t}
\begin{Condition} \label{cond:derivative:t}
	Keeping the same constant $\gamma_0$ as in Condition \ref{cond:covFunctCondition}, there exists another constant $\rho_0 > 0$ such that for all $f,g \in \functionSpace$ satisfying $0 <\|f - g\|  \leq \rho_0$, we have
	\[
	\sup_{\theta \in \Theta} 
	\frac{
		\left|K_{\theta}(f) - K_{\theta}(g)
	\right| }{\|f - g\|} 
	\ \triangleleft \ \frac{1}{1+\min(\|f\|,\|g\|)^{\gamma_0} }.
	\]
\end{Condition}

To obtain the asymptotic normality of the ML estimator for exact data, i.e. Theorem~\ref{theo:normality}, we also required some control over the partial derivatives of orders two and three of the function $K_{\theta}$ (see Condition~\ref{cond:d3covFunctCondition}). In order to pass down this condition to the approximated data we will require, in the same manner as Condition \ref{cond:derivative:t}, a local control over the rate of change of the partial derivatives of the covariance function.
\begin{Condition} \label{cond:derivative:t:der:theta}
	Assume that Conditions \ref{cond:dcovFunctCondition} holds and keep the same constant $\gamma_1 > 0$. Then there exists another constant $\rho_1 > 0$ such that for all $i \in \{1, \ldots, p\}$, and for all $f,g \in \functionSpace$ such that $0 < \|f-g\| \leq \rho_1$, we have
	\begin{equation*}
		\sup_{\theta \in \Theta}
		\frac{1}{\|f - g\|}
		\left|
		\frac{\partial K_\theta}{\partial \theta_{i}} (f)
		-
		\frac{\partial K_\theta}{\partial \theta_{i}}(g)
		\right| \ \triangleleft \ \frac{1}{1+\min(\|f\|,\|g\|)^{\gamma_1} }.
	\end{equation*}
\end{Condition}

\begin{Condition} \label{cond:derivative3:t:der:theta}
	Assume that Conditions \ref{cond:d3covFunctCondition} holds and keep the same constant $\gamma_2 > 0$. Then there exists another constant $\rho_2 > 0$ such that for all $q \in \{2,3\}$, for all $i_1, \ldots, i_q \in \{1, \ldots, p\}$, and for all $f,g \in \functionSpace$ such that $0 < \|f-g\| \leq \rho_2$, we have
	\begin{equation*}
		\sup_{\theta \in \Theta}
		\frac{1}{\|f - g\|}
		\left|
		\frac{\partial^q K_\theta}{\partial \theta_{i_1}\cdots \partial \theta_{i_q}} (f)
		-
		\frac{\partial^q K_\theta}{\partial \theta_{i_1}\cdots \partial \theta_{i_q}}(g)
		\right| \ \triangleleft \ \frac{1}{1+\min(\|f\|,\|g\|)^{\gamma_2} }.
	\end{equation*}
\end{Condition}

\paragraph{The case of isotropic processes.}
As in Section~\ref{sec:asympProperties}, we can give some simplified versions of Conditions \ref{cond:derivative:t}, \ref{cond:derivative:t:der:theta} and \ref{cond:derivative3:t:der:theta} whenever the underlying process is assumed to be isotropic. The adaptations of these conditions are given in Appendix~\ref{app:CondIsotropy} (Conditions~\ref{cond:derivative:t:Iso}, \ref{cond:derivative:t:der:theta:Iso}, and \ref{cond:derivative3:t:der:theta:Iso} respectively).

\subsection{Asymptotic consistency}
\label{sec:asympPropertiesApprox:subsec:consistency}
We start by establishing the consistency of the ML estimator. To prove Theorem~\ref{theo:consistencyApprox}, we make use of the technical Lemma~\ref{lemma:lambdaMaxApprox} which will allow us to control $\lambda_{\max}(R_{N_n,\theta} - R_{\theta} )$.
\begin{lemma} 
	\label{lemma:lambdaMaxApprox}
	Assume that Conditions~\ref{cond:distant}, \ref{cond:orthormalBasis}, and~\ref{cond:approx:inputs} hold. If Condition~\ref{cond:derivative:t} is verified, then 
	$$\sup_{\theta \in \Theta}\lambda_{\max }\left(R_{\theta} - R_{N_n,\theta}\right) \xrightarrow[\substack{n \to \infty} ]{} 0.$$
 If Condition~\ref{cond:derivative:t:der:theta} is verified, then we have
    $$\sup_{\theta \in \Theta}\max_{i=1,\ldots,p}\lambda_{\max}\left(\frac{\partial}{\partial \theta_{i}} (R_{\theta}-R_{N_n,\theta})\right) \xrightarrow[\substack{n \to \infty }]{} 0.$$
\end{lemma}	
Next, we introduce Theorem \ref{theo:consistencyApprox}. As mentioned before, this theorem guarantees that the numerical results obtained using only an approximation of the input functions will still converge correctly.
\begin{theorem}
	\label{theo:consistencyApprox}
	Let $\thetaMLapprox \in \Theta$ be the maximum likelihood estimator obtained by minimizing the negative log Gaussian likelihood of the approximated data in~\eqref{eq:negLogGaussianLikeApprox}. Under Conditions~\ref{cond:distant} to~\ref{cond:diffK}, as well as
 Condition \ref{cond:dcovFunctCondition},
 Condition \ref{cond:approx:inputs}, Condition~\ref{cond:derivative:t} and Condition~\ref{cond:derivative:t:der:theta}, we have
	\begin{equation*}
		\thetaMLapproxNn \xrightarrow[\substack{n \to \infty }]{\mathbb{P}} \theta_0.
	\end{equation*}
\end{theorem}

\begin{proof}[Proof of Theorem \ref{theo:consistencyApprox}]
	%We have $\thetaMLapprox \in \argmin_{\theta \in \Theta} \negloglikeFun_{n,N}(\theta)$ with
	%\begin{equation*}
		%	\negloglikeFun_{n}(\theta) = \frac{1}{n} \ln|R_\theta| + \frac{1}{n} y^\top R_\theta^{-1} y.
		%\end{equation*}
	
	Lemma~\ref{lemma:lambdaMaxApprox} enables us to extend Conditions~		\ref{cond:infEingenVals} and \ref{cond:diffK} to noisy inputs:
	\begin{equation*}
		%\sup_{\theta \in \Theta}
		\sup_{\theta \in \Theta}\lambda_{\min}(R_{N_n, \theta}) \geq c, \qquad c > 0,
	\end{equation*}
	\begin{equation*}
		\liminf_{n \to \infty} \inf_{\|\theta-\theta_0\| \geq \alpha} \frac{1}{n} \sum_{i,j = 1}^{n} [K_{\theta}(f_i^{N_n},f_j^{N_n})-K_{\theta_0}(f_i^{N_n},f_j^{N_n})]^2 > 0, \qquad \forall \alpha >0.
	\end{equation*}
	From Lemmas~\ref{lemma:lambdaMax} and \ref{lemma:lambdaMaxApprox}, we have that 
	\begin{equation*}
		\sup_{\theta \in \Theta} \lambda_{\max} (R_{N_n, \theta})
		\quad \mbox{and} \quad
		\sup_{\theta \in \Theta} \max_{i = 1, \ldots, p} \lambda_{\max} \left(\frac{\partial R_{N_n}}{\partial \theta_i}(\theta) \right)
	\end{equation*}
	are bounded as $n \to \infty$. Then, we can follow the proof of Theorem~\ref{theo:consistency} to conclude the proof.
\end{proof}

\subsection{Asymptotic normality}
\label{sec:asympPropertiesApprox:subsec:normality}
For the asymptotic normality to hold, one requires some control over the partial derivatives of orders two and three of the modified negative log-likelihood. To pass down this property to the approximate data case, we will need to impose control over the rate of change of said derivatives; this will be done thanks to Conditions~\ref{cond:derivative:t:der:theta} and \ref{cond:derivative3:t:der:theta} through Lemma~\ref{lemma:lambdaMaxApprox'}.
\begin{lemma}
\label{lemma:lambdaMaxApprox'}
    Assume that Conditions~\ref{cond:distant}, \ref{cond:orthormalBasis}, \ref{cond:approx:inputs}, \ref{cond:derivative:t:der:theta}, and~\ref{cond:derivative3:t:der:theta} hold. Then for $q \in \{1,2,3\}$, $i_1, \ldots, i_q \in \{1,\ldots,p\}$, we have
	$$\sup_{\theta \in \Theta}\max_{i=1,\ldots,p}\lambda_{\max}\left(\frac{\partial}{\partial \theta_{i_1}}\cdots \frac{\partial}{\partial \theta_{i_q}} (R_{\theta}-R_{N_n,\theta})\right) \xrightarrow[\substack{n \to \infty }]{} 0.$$
\end{lemma}
We now provide the asymptotic normality of the estimator for the approximated data.
\begin{theorem}
	\label{theo:normalityApprox}
	Let $\asympConMatapprox$ be the $p \times p$ asymptotic covariance matrix defined by
	\begin{equation*}
		(\asympConMatapprox)_{i,j} = \frac{1}{2n} \tr\bigg(R_{N,\theta_0}^{-1} \frac{\partial R_{N}}{\partial \theta_i}(\theta_0) R_{N,\theta_0}^{-1} \frac{\partial R_{N}}{\partial \theta_j}(\theta_0)\bigg).
	\end{equation*}
	Under Conditions~\ref{cond:distant} to~\ref{cond:d3covFunctCondition}, and Conditions~\ref{cond:approx:inputs} to~\ref{cond:derivative3:t:der:theta}, we have
	\begin{equation}
	    \label{eq:asympNormApprox}
		\sqrt{n} \asympConMatapproxNn^{1/2} (\thetaMLapproxNn - \theta_0) \xrightarrow[n \to \infty]{\negloglikeFun} \mathcal{N}(0, I_p). 
	\end{equation}
	Furthermore,
	\begin{equation}
		0 < \liminf_{n \to \infty} \lambda_{\min}(\asympConMatapproxNn) \leq \limsup_{n \to \infty} \lambda_{\max}(\asympConMatapproxNn) < \infty.
		\label{eq:posEigValsApprox}
	\end{equation}
\end{theorem}

\begin{proof}
	From Lemmas~\ref{lemma:lambdaMax} and~\ref{lemma:lambdaMaxApprox}, and Condition~\ref{cond:infEingenVals}, we show~\eqref{eq:posEigValsApprox} similarly as in the proof of Theorem~\ref{theo:normality}. The arguments of the remaining of the proof are similar to the ones considered in the proof of Theorem~\ref{theo:normality}.
\end{proof}
As observed in \eqref{eq:asympNormApprox}, the renormalization matrix $\asympConMatapproxNn$ depends on the approximation order $N_n$. This dependency on $N_n$ poses some challenges in computing confidence ensembles for the estimated parameters $\theta_0$. One way to solve this issue is to prove that the covariance matrix for the approximated data, that is $\asympConMatapproxNn$, is asymptotically close to the asymptotic covariance matrix $\asympConMat$ which is the aim of Lemma \ref{lemma:convConMat}.
\begin{lemma}
\label{lemma:convConMat}
Consider the same conditions as in Theorem~\ref{theo:normalityApprox}. We have the following convergence of the covariance matrix for the approximated data, denoted by $\asympConMatapproxNn$, toward the theoretical covariance matrix $\asympConMat$: as $n \xrightarrow[]{} \infty$,
    \[
    \asympConMatapproxNn - \asympConMat 
    \xrightarrow[n \to \infty]{} 0.
    %= o_{\textrm{a.s.}}\left(\frac{1}{n}\right),
    \]
    %where we write $X_n = o_{\textrm{a.s.}}\left(\frac{1}{n}\right)$ whenever $nX_n(\omega) \xrightarrow[n \to \infty]{\normalfont{a.s.}} 0$ for some sequence of random variables $(X_n)_{n \in \mathbb{N}^*}$
\end{lemma}

Using Lemma~\ref{lemma:convConMat}, we can obtain an extension of Theorem~\ref{theo:normalityApprox} renormalized using the theoretical covariance matrix $\asympConMat$ instead.
\begin{corollary}
\label{theo:normalityApprox2}
Under the hypothesis of Theorem~\ref{theo:normalityApprox}, we have
\begin{equation}
		\sqrt{n} \asympConMat^{1/2} (\thetaMLapproxNn - \theta_0) \xrightarrow[n \to \infty]{\negloglikeFun} \mathcal{N}(0, I_p). 
	    \label{eq:asympNormApprox2}
	\end{equation}
\end{corollary}
\begin{proof}
    According to Theorem \ref{theo:normalityApprox}, when $n \xrightarrow[]{} \infty$, we have the convergence
    \begin{equation}
		\sqrt{n} \asympConMatapproxNn^{1/2} (\thetaMLapproxNn - \theta_0) \xrightarrow[n \to \infty]{\negloglikeFun} \mathcal{N}(0, I_p). 
	\end{equation}
    Moreover, for any $n \in \mathbb{N}^*$, it holds
    \[
    \sqrt{n} \asympConMat^{1/2} (\thetaMLapproxNn - \theta_0) = \sqrt{n} \asympConMatapproxNn^{1/2} (\thetaMLapproxNn - \theta_0) + \sqrt{n} (\asympConMat^{1/2} -\asympConMatapproxNn^{1/2}) (\thetaMLapproxNn - \theta_0).
    \]
    However, according to Lemma \ref{lemma:convConMat}, we have
    \[
    \asympConMat^{1/2}-\asympConMatapproxNn^{1/2} \xrightarrow[n \to \infty]{} 0. 
    \]
    Then applying Slutsky theorem, we can conclude the proof.
\end{proof}

\section{Analytical examples and numerical illustrations}
\label{sec:examplesAsympPropertiesApprox}

\subsection{Analytical examples}
\label{sec:examplesAsympPropertiesApprox:subsec:examples}
Here, we give some examples of the application of the theorems mentioned in Section~\ref{sec:asympPropertiesApprox}. We consider the Hilbert space $\mathcal{H} = L^2([0,1])$ endowed with its usual inner product as well as the subspace
\begin{equation*}
	E_{\omega}
	:=
	\left\{f \in \mathcal{H}; \forall u, v \in [0, 1], |f(u) - f(v)| \leq \omega(|u-v|) 
	\right\},
\end{equation*}
of functions of $\mathcal{H}$ which admits $\omega : [0,1] \to [0,1]$ as a global modulus of continuity. Recall that $\omega(t) \to 0$ as $t \to 0$. Examples of such spaces include spaces of continuous functions, spaces of Lispchitz continuous functions, or spaces of Hölder continuous functions.

Recall that we denote by $\|\cdot\|$ the norm on $\functionSpace$ induced by the usual inner product, that is for any $f \in \functionSpace$
\[
\|f\|^2 = \int_{0}^1{f^2(x)dx}. 
\]
We also denote by $\|\cdot\|_{\infty}$ the uniform norm, that is for $f: [0,1] \to \mathbb{R}$, $\|f\|_{\infty} = \sup_{x \in [0,1]}|f(x)|$. Finally, we assume that the input functions $f_n$, $n \in \mathbb{N}$, are all elements of the linear subspace $E_{\omega}$.

\begin{example}[Sampling] \label{example:sampling}
	Let $f_i^N$ be the piecewise constant approximation of $f_i$ given by 
	\begin{equation*}
		f_i^N(x) = f_i\left(\frac{\overline{(N-1)x}}{N-1}\right) \qquad x \in [0,1],
	\end{equation*}
	with $\overline{(N-1)x}$  the integer part of $(N-1)x$. Then,
	\begin{equation*}
		\|f_i - f_i^{N_n}\| \leq \|f_i - f_i^{N_n}\|_{\infty} \leq \omega\left(\frac{1}{N_n-1}\right),
	\end{equation*}
	which verifies Condition~\ref{cond:approx:inputs} for any $N_n \to \infty$.
\end{example}

\begin{example}[Random sampling with time-independent samples]
\label{example:rndSampling}
	
	Consider random sampling where samples are fixed for all $i \in \mathbb{N}$. Let $U_{1}, \ldots, U_{N} \underset{iid}{\sim} \mathcal{U} ([0,1])$ and denote the order process by $0 = U_{(0)} < U_{(1)} < \cdots < U_{(N)}$. The random sampling approximation is given by
	\begin{equation*}
		f_i^N(x) = f_i\left(U_{(I_N(x))}\right),
	\end{equation*}
	with $I_N(x) = \max \left\{j \in \{0, \ldots, N\}; U_{(j)} \leq x\right\}$.
	
	As $N \to \infty$,  
	we have $\max_{j=1, \ldots, N} (U_{(j)} - U_{(j-1)}) \to 0$ almost surely. Hence, for any $N_n \to \infty$,        Condition~\ref{cond:approx:inputs} is verified almost surely, by the same arguments as in Example \ref{example:sampling}. 
\end{example}

\begin{example}[Random sampling with time-varying samples]
\label{example:rndSamplingTimeVarying}
	Consider now that $U_{1}^i, \ldots, U_{N}^i \underset{iid}{\sim} \mathcal{U} ([0,1])$. Denote the order process by $0 = U_{(0)}^i < U_{(1)}^i < \cdots < U_{(N)}^i$ for all $i \in \mathbb{N}$. The random sampling approximation is given by        \begin{equation*}
		f_i^N(x) = f_i\left(U_{(I_N^i (x))}^i\right),
	\end{equation*}
	with $I_N^i(x) = \max \left\{j \in \{0, \ldots, N\}; U_{(j)}^i \leq x\right\}$. For all $n$ and for all $i = 1, \ldots, n$, it holds
	\begin{align*}
		\|f_i - f_i^{N_n}\| 
		<
		\sup_{\substack{f \in E_\omega \\ x \in [0,1]}} \bigg|f(x) - f\left(U_{(I_{N_n}^i(x))}^i\right)\bigg| 
		\leq \max_{j = 1, \ldots, N_n } \omega (U_{(j)}^i - U_{(j-1)}^i). %\xrightarrow[a.s.]{} 0.
	\end{align*}
	Let $\delta > 0$ and denote by $F^{(N)}$ the empirical cumulative distribution function of $U_{1}, \ldots, U_{N}$. Then, since $F^{(N)}$ is almost surely piecewise constant in between the observations $U_1^1, \ldots, U_N^1$, we deduce that  $$\left\lbrace \exists j \in \{1, \ldots, N_n\},  U_{(j)}^{1} - U_{(j-1)}^{1}  > \delta\right\rbrace \subset \left\lbrace \sup_{x \in [0,1]}\left|F^{(N)}(x) - x\right| \geq \frac{\delta}{2} \right\rbrace,$$ and therefore
	\begin{equation*}
		\mathbb{P}\left(\max_{j=1,\ldots,N_n} \left(U_{(j)}^{1} - U_{(j-1)}^{1}\right) > \delta\right) \le \mathbb{P}\left(\sup_{x \in [0,1]}\left|F^{(N)}(x) - x\right| \geq \frac{\delta}{2}\right).
	\end{equation*}
	Moreover, letting $\omega^{-1}(t) = \inf \{s \geq 0; \omega(s) \geq t\}$,
	\begin{align*}
		\mathbb{P}\left(\max_{i=1, \ldots, n}  \|f_i - f^{N_n}_i\| > \delta\right)
		&\leq n \mathbb{P}(\|f_1 - f^{N_n}_1\| > \delta)
		\\
		&\leq n\mathbb{P}\left(\omega\left(\max_{j =1, \ldots, N_n} \left(U_{(j)}^1 - U_{(j-1)}^1\right) \right) \geq \delta\right)
		\\
		& \le n\mathbb{P} \left(\max_{j = 1, \ldots, N_n} \left(U_{(j)} - U_{(j-1)}\right) \geq \omega^{-1}(\delta)  \right)
		\\
		&\leq n\mathbb{P}\left(\|F^{(N_n)} - I_d \|_{\infty} \geq \frac{\omega^{-1}(\delta)}{2}\right).
		%\\
		%&\leq 2ne^{-2N_n(\omega^{-1}(\delta))^2} \qquad \qquad \text{(Dvoretzky-Kiefer-Wolfowitz inequality \cite{Massart1990DKWIneq})}.
	\end{align*}
	Then, by using the Dvoretzky-Kiefer-Wolfowitz inequality \cite{Massart1990DKWIneq}, we have
	\begin{equation*}
		\mathbb{P}\left(\max_{i=1, \ldots, n}  \|f_i - f^{N_n}_i\| > \delta\right)
		\leq 2n e^{-N_n(\omega^{-1}(\delta))^2/2}.
	\end{equation*}
	Hence, if $\frac{N_n}{\log n} \to \infty$ as $n \to \infty$, then we have 
	\begin{equation*}
		\max_{i=1, \ldots, n}  \| 
		f_i - f^{N_n}_i
		\| \xrightarrow[n \to \infty]{\mathbb{P}} 0. 
	\end{equation*}
	Then, there exists a sequence $\xi_n \xrightarrow[n \to \infty]{} 0$ such that 
	\begin{equation*}
		\mathbb{P}\left( \max_{i= 1, \ldots, n} \|
		f_i - f^{N_n}_i
		\| > \xi_n \right) \xrightarrow[n \to \infty]{} 0.
	\end{equation*}

	Denote 
	\begin{equation*}
		\overline{f}_i^{N_n}
		= \begin{cases}
			f_i^{N_n}, & \text{ if } \max_{i= 1, \ldots, n} \|f_i - f^{N_n}_i\| \leq \xi_n,
			\\
			f_i, & \text{otherwise}.
		\end{cases}
	\end{equation*}
	Let ${\widehat{\theta}}_{n, N_n}$ and $\overline{\widehat{\theta}}_{n, N_n}$ be the estimated parameters from $f_i^{N_n}$ and $\overline{f}_i^{N_n}$, respectively. Then, almost surely
	\begin{equation*}
		\max_{i= 1, \ldots, n} \|f_i - \overline{f}^{N_n}_i\| \leq \xi_n,
	\end{equation*}
	which implies, for almost all samplings, both the consistency and normality of $\overline{\widehat{\theta}}_{n, N_n}$ through Theorems~\ref{theo:consistencyApprox} and~\ref{theo:normalityApprox} as long as the rest of the hypotheses are satisfied. Since $\mathbb{P}(\overline{\widehat{\theta}}_{n, N_n} =  {\widehat{\theta}}_{n, N_n}) \to 1$ as $n \to \infty$, both guarantees also hold for ${\widehat{\theta}}_{n, N_n}$ for almost all samplings.
	
\end{example}

\begin{example}[Bernstein approximation]
\label{example:Bernstein}
For any $N$, and any $k \in \{1, \ldots, N\}$, denote by $B_{N, k} : x \mapsto {\begin{pmatrix} N \\ k \end{pmatrix}} x^k (1-x)^{N-k}$ the Bernstein polynomials of degree $N$ with binomial coefficients {$\begin{pmatrix} N \\ k \end{pmatrix}$}. The Bernstein approximation $f_i^N$ of $f_i$ is then given by $f_i^N(x) = \sum_{k=0}^{N} f\left(\frac{k}{N}\right) B_{N,k} (x)$ for all $x \in [0,1]$. In this case (\cite{Schatzman2002numerical}, Theorem 5.3.2)
states	\begin{equation*}
		\|f_i - f_i^{N_n}\|
		\leq \frac{9}{4} \omega\left(\frac{1}{\sqrt{N_n}} \right),
	\end{equation*}
	which satisfies Condition~\ref{cond:approx:inputs} for any $N_n \xrightarrow[n \to \infty]{} \infty$.
\end{example}

\begin{example}[Cubic spline approximation]
\label{example:Splines}
	Let $\xi_1, \ldots, \xi_N$ be the knots of the spline approximation such that $0 = \xi_1 < \cdots < \xi_N = 1$ and $\Delta \xi_j = \xi_{j+1} - \xi_j$ for all $j \in \{1, \ldots, N-1\}$. Assume that $\max_{j = 1, \ldots, N-1} \Delta \xi_i \to 0$ as $N \to \infty$. For all $i \in \mathbb{N}$, consider $f_i \in \mathcal{C}^{4}([0,1])$ with $\mathcal{C}^{4}$ the set of four times continuously differentiable functions. Define $f_i^{(4)}$ as the $4$-th derivative of $f_i$, and $f_i^{N}$ as the cubic spline approximation of $f_i$. Then (\cite{Hall1968ErrorBoundsSplineInterpolation}, Theorem 1),
	\begin{equation*}
		\|f_i - f_i^{N_n}\| 
		\leq \|f_i - f_i^{N_n}\|_\infty
		\triangleleft \left(\max_{j = 1, \ldots, N_n-1} \Delta \xi_j\right)^{4} \|f^{(4)}\|_{\infty},
	\end{equation*}
	which verifies Condition~\ref{cond:approx:inputs} for any $N_n \to \infty$.
\end{example}
Note that there exist many different error bounds for spline approximations under a variety of hypotheses (for a more in-depth analysis, see \cite{Rosen1971MinimumErrorBoundsSplines}) which would lead to different but similar cases as Example~\ref{example:Splines}. 

\subsection{Numerical illustrations}
\label{sec:examplesAsympPropertiesApprox:subsec:numIllustrations}

We now numerically verify Theorems~\ref{theo:consistencyApprox} and~\ref{theo:normalityApprox} for the examples in Section~\ref{sec:examplesAsympPropertiesApprox:subsec:examples}.  Due to computational challenges involved in both data generation and covariance parameter estimation, we have limited the asymptotic analysis to $n \leq 160$ and $N_n \leq 160$. This choice is based on the fact that conducting these tasks for $n > 160$ would require more than dozens of hours. Furthermore, as observed in Figures~\ref{fig:consistencyApprox} and~\ref{fig:normalityApprox}, it is worth noting that convergence results become visible when $n = 160$. Our experiments have been executed on an 11th Gen Intel(R) Core(TM) i5-1145G7 @ 2.60GHz 16 Gb RAM. The Python source codes are available on the GitHub repository: \url{https://github.com/anfelopera/fGPs}. The different plots have been generated using the Python toolbox Seaborn \cite{Waskom2021Seaborn}.
\\
\\
We consider inputs given by the functional form:
\begin{equation}
    f_t(x) = 2t + \cos(U_0(t) x) e^{-U_1(t) x},
    \label{eq:fexamples}
\end{equation}
with $x \in [0,1]$, and $(U_0(t))_{t \in \mathbb{N}},(U_1(t))_{t \in \mathbb{N}}$ two i.i.d. sequences of i.i.d. random variables sampled from a uniform distribution $U[0, 1]$. While $U_0$ introduces oscillations in $f$, $U_1$ controls its decay rate. The parameter $t \in \mathbb{N}$ is introduced to satisfy Condition~\ref{cond:distant} related to the increasing-domain asymptotics.
\\
\\
We sample $n$ random inputs $f_1, \ldots, f_n$ from~\eqref{eq:fexamples}. These inputs are then used to compute the corresponding outcomes $y_i = Y(f_i)$, where $Y(f_i)$ are obtained as a random realization of a GP with the SE kernel in~\eqref{eq:statkernel} and covariance parameters $\theta_0 = (0.5, 0.1)$. These values have been arbitrarily chosen for illustrative purposes. The resulting collection of tuples $(f_i, y_i)_{i=1, \ldots, n}$ form the training dataset. For the asymptotic analysis, we consider $n \geq 20$.
\\
\\
For a given value of $n$, we repeat the aforementioned strategy for generating $10^3$ distinct replicates of the training dataset. Then, for each replicate, we train a new GP with a SE kernel. We consider approximations based on sampling, random sampling with time-independent samples or with time-varying samples, and Bernstein polynomials. For the latter three cases, the $L_2$-norm is approximated at $d = 500$ equispaced values of $x \in [0,1]$. For random sampling with time-varying samples, due to the condition $\frac{N_n}{\log n} \to \infty$ discussed in Example~\ref{example:rndSamplingTimeVarying}, we fix $N_n = 160$ independently of the value of $n$. Finally, the covariance parameters $\theta = (\theta_1, \theta_2)$ are estimated via ML in a multi-start context, with five initial values randomly sampled from $[0.1, 10]^2$. To do so, the gradient-based modified Powell algorithm available in the Python toolbox SciPy~\cite{2020SciPy-NMeth} has been used. This entire procedure results in $10^3$ estimations of $\theta_1$ and $\theta_2$, which will be used for assessing Theorems~\ref{theo:consistencyApprox} and~\ref{theo:normalityApprox}. For the latter, we focus on Corollary~\ref{theo:normalityApprox2}.
\begin{figure}[t!]
    \centering
    \includegraphics[width = 0.325\linewidth]{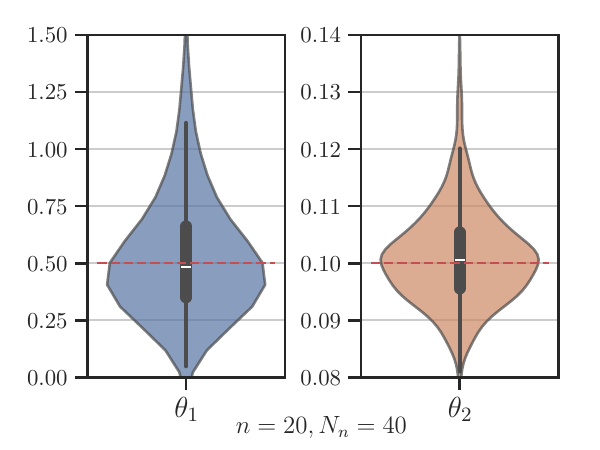}
    \includegraphics[width = 0.325\linewidth]{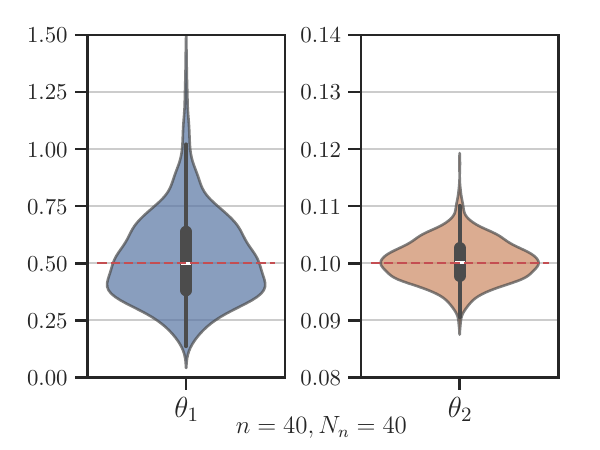}
    \includegraphics[width = 0.325\linewidth]{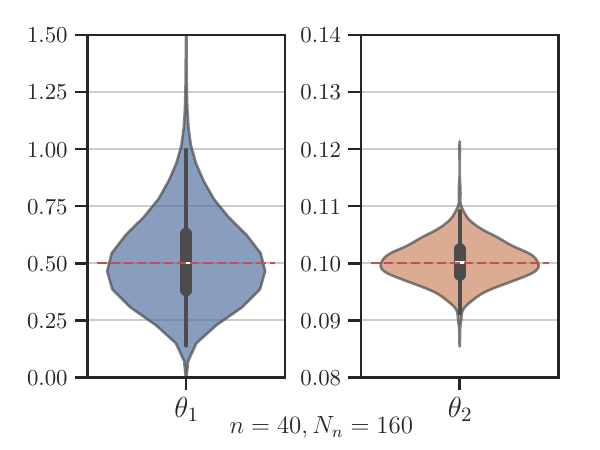}
    \vspace{-1ex}

    \includegraphics[width = 0.325\linewidth]{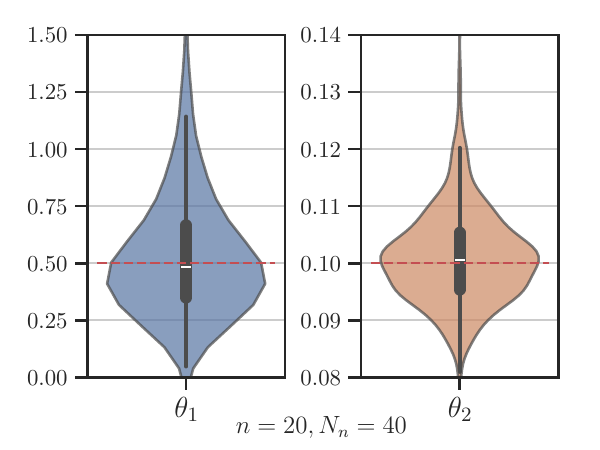}
    \includegraphics[width = 0.325\linewidth]{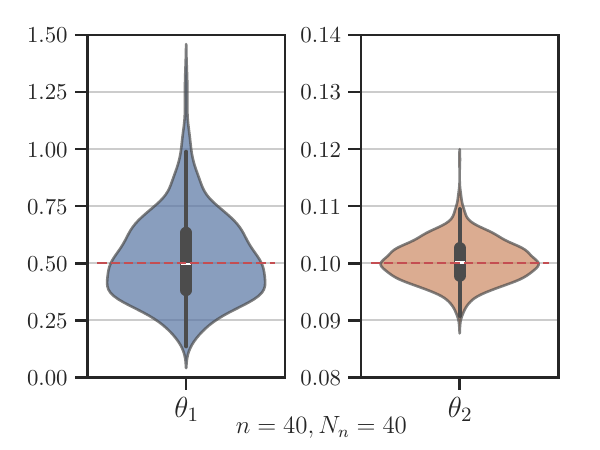}
    \includegraphics[width = 0.325\linewidth]{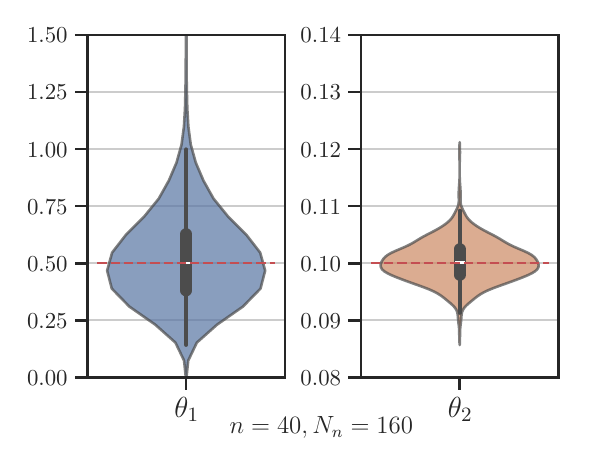}
    \vspace{-1ex}

    \includegraphics[width = 0.325\linewidth]{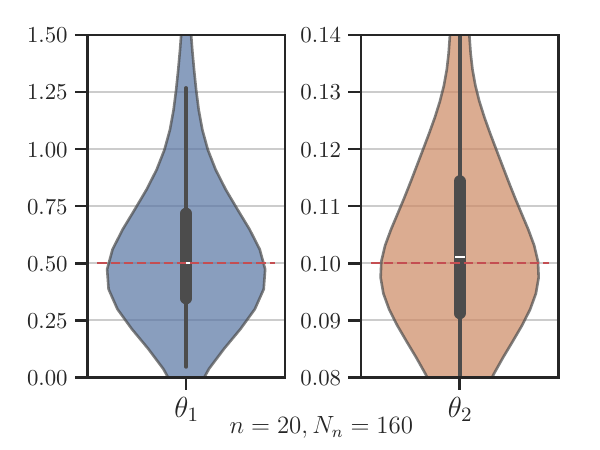}
    \includegraphics[width = 0.325\linewidth]{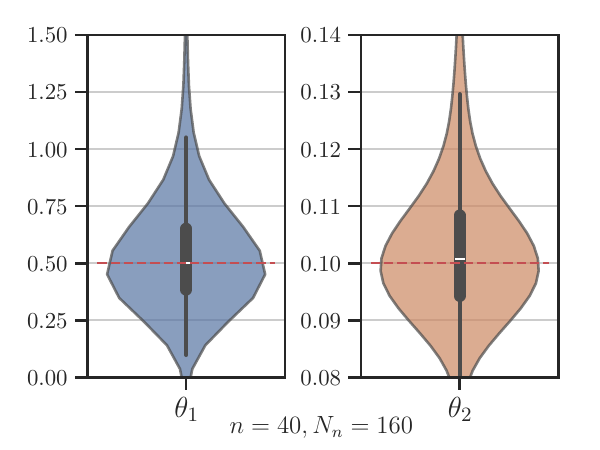}
    \includegraphics[width = 0.325\linewidth]{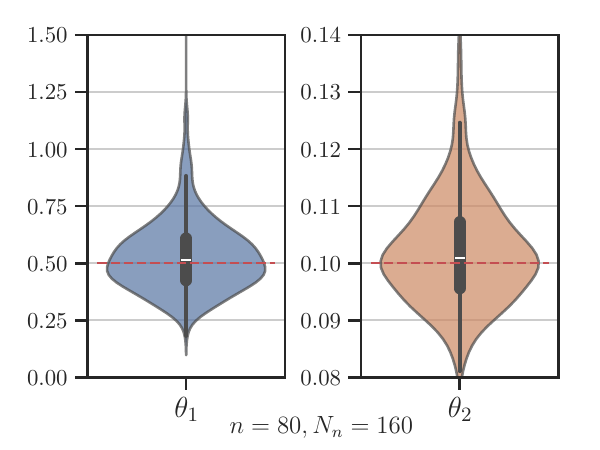}
    \vspace{-1ex}

    \includegraphics[width = 0.325\linewidth]{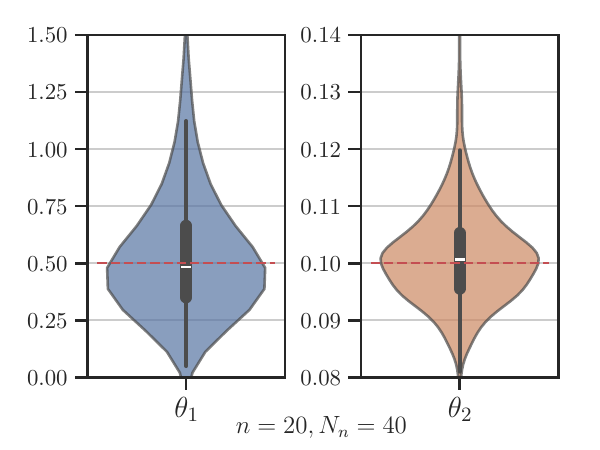}
    \includegraphics[width = 0.325\linewidth]{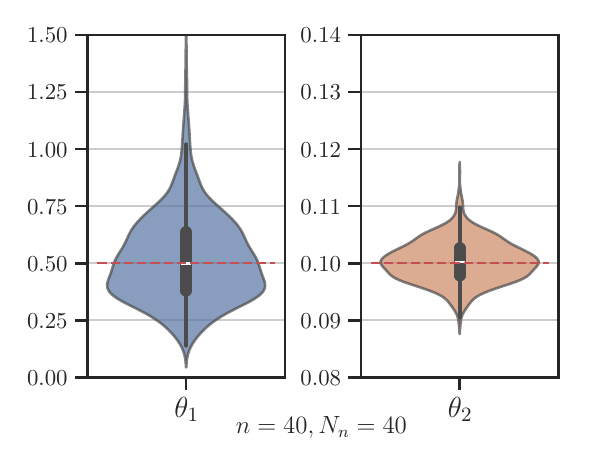}
    \includegraphics[width = 0.325\linewidth]{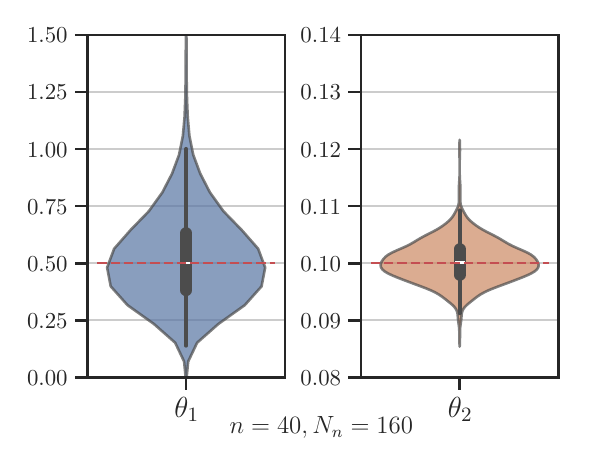}
    \caption{Violin plots of the estimated covariance parameters $\widehat{\theta} = (\widehat{\theta}_1, \widehat{\theta}_2)$ for Examples~\ref{example:sampling}-\ref{example:Bernstein} (from top to bottom: sampling, random sampling with time-independent samples, random sampling with time-varying samples, Bernstein), and for different values of $n \geq 20$ and $N_n \geq 40$ (results shown by rows). %For the case of random sampling with time-varying samples (third row), due to the condition $\frac{N_n}{\log n} \to \infty$, we let $N_n = 160$.
    The panel illustrates the estimation results (boxplot and KDE plot) of the variance $\widehat{\theta}_1$ and length-scale $\widehat{\theta}_2$ computed over $10^3$ replicates. The ground truth $\theta_0$ is represented by a horizontal red dashed line.}
    \label{fig:consistencyApprox}
\end{figure}
\\
\\
Figure~\ref{fig:consistencyApprox} shows the violin plots illustrating both the boxplots and the kernel density estimate (KDE) plots for $\widehat{\theta}_1$ and $\widehat{\theta}_2$. We observe that, for all four cases, as $n$ and $N_n$ increase, the distributions of the estimators become closer to the Gaussian ones with median values closer to the ground truth $\theta_0 = (0.5, 0.1)$. It is noteworthy that, due to the complexity of the random sampling with time-varying samples, a larger value of $n$ has been needed for more accurate median values with contracted distributions around the median. Overall, the results support the verification of Theorem~\ref{theo:consistencyApprox} related to the asymptotic consistency of the ML estimator.
\begin{figure}[t!]
    \centering
    \includegraphics[width = 0.3\linewidth]{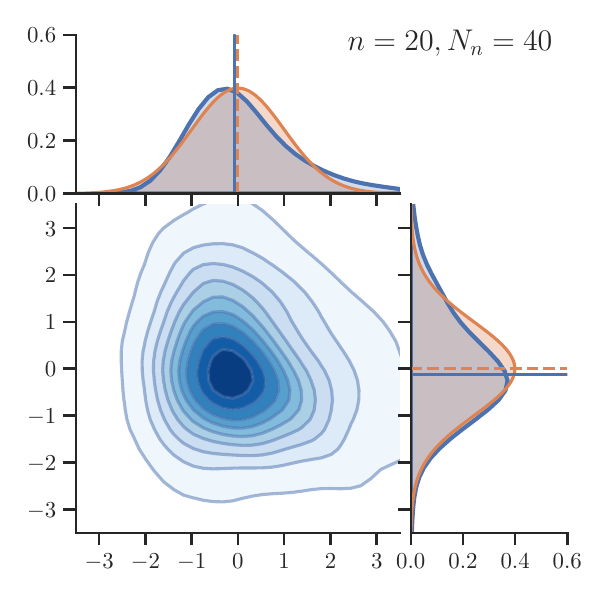}
    \includegraphics[width = 0.3\linewidth]{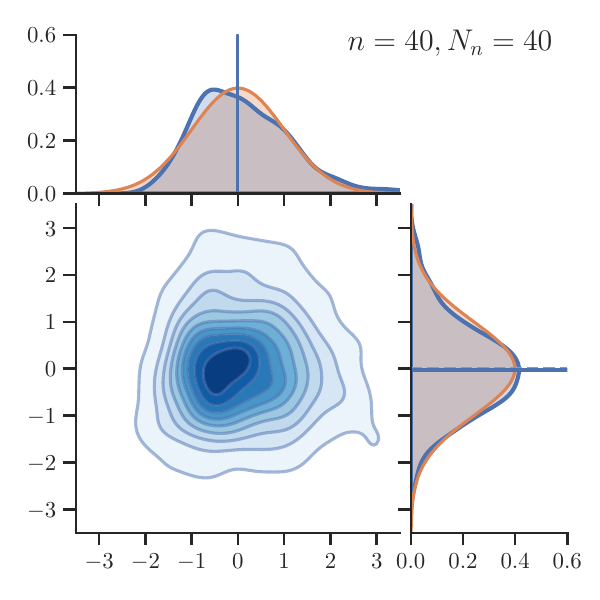}
    \includegraphics[width = 0.3\linewidth]{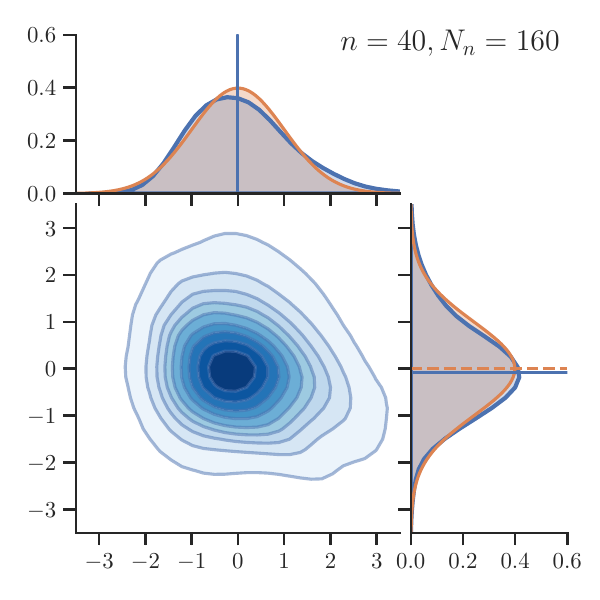}
    \vspace{-1.5ex}

    \includegraphics[width = 0.3\linewidth]{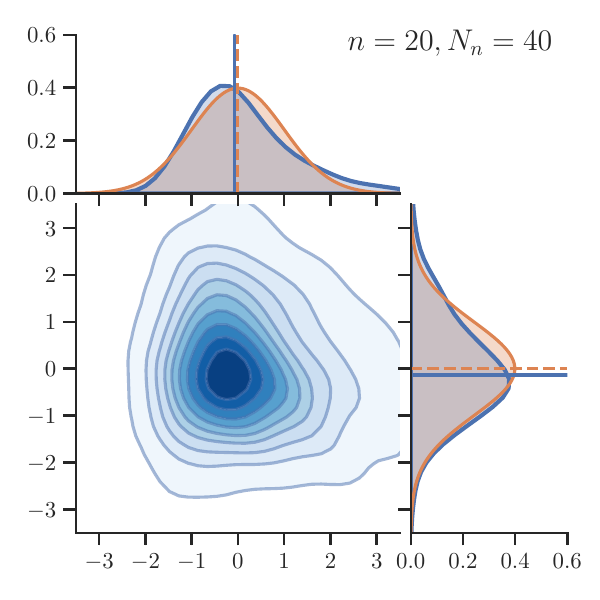}
    \includegraphics[width = 0.3\linewidth]{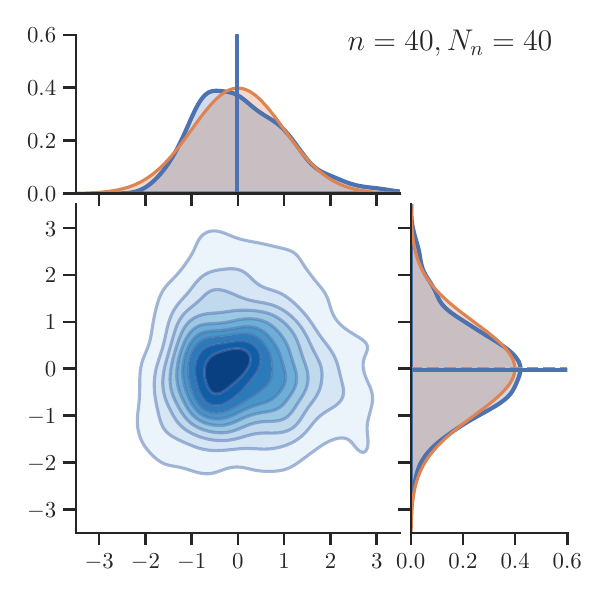}
    \includegraphics[width = 0.3\linewidth]{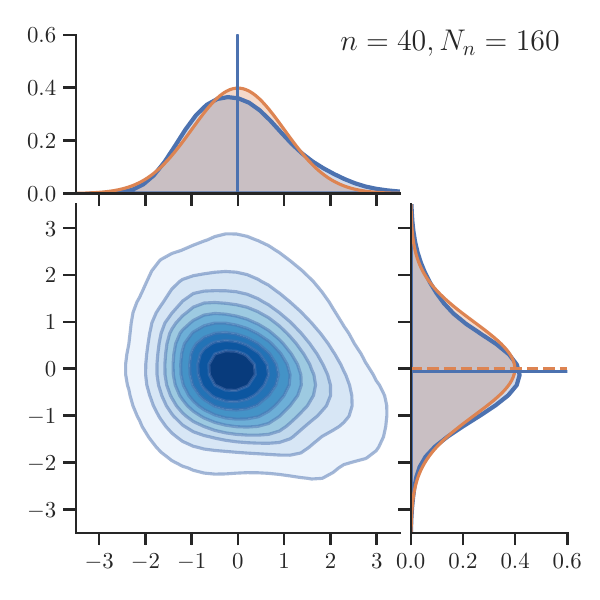}
    \vspace{-1.5ex}

    \includegraphics[width = 0.3\linewidth]{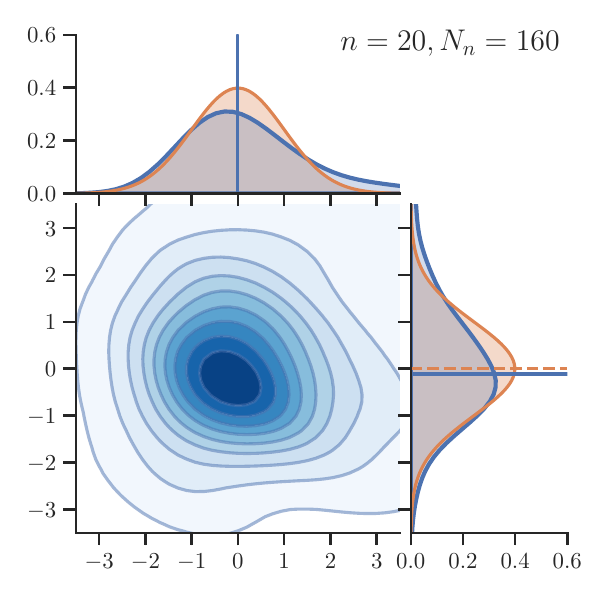}
    \includegraphics[width = 0.3\linewidth]{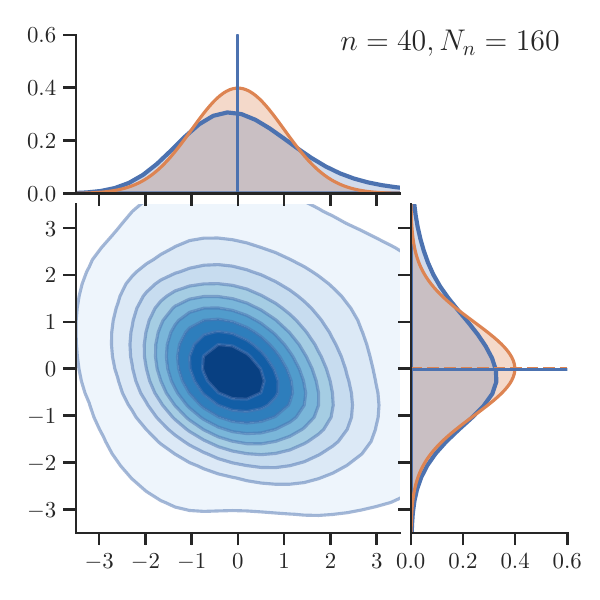}
    \includegraphics[width = 0.3\linewidth]{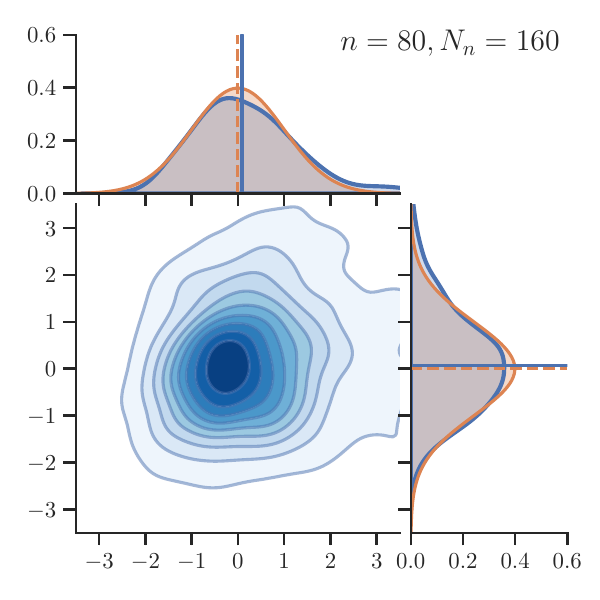}
    \vspace{-1.5ex}
    
    \includegraphics[width = 0.3\linewidth]{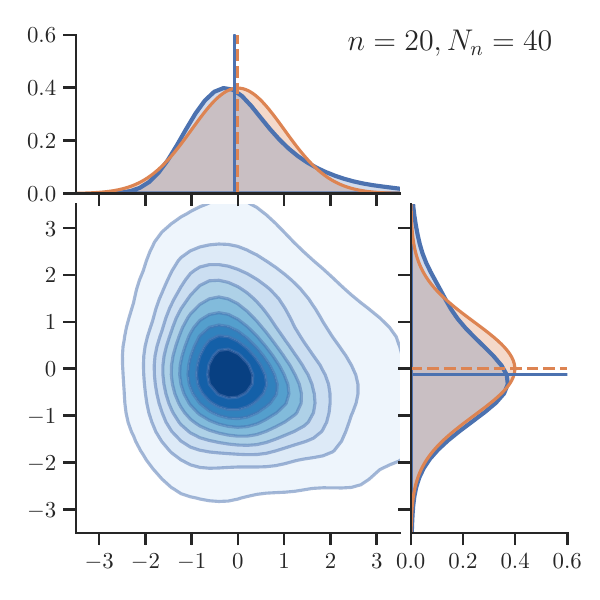}
    \includegraphics[width = 0.3\linewidth]{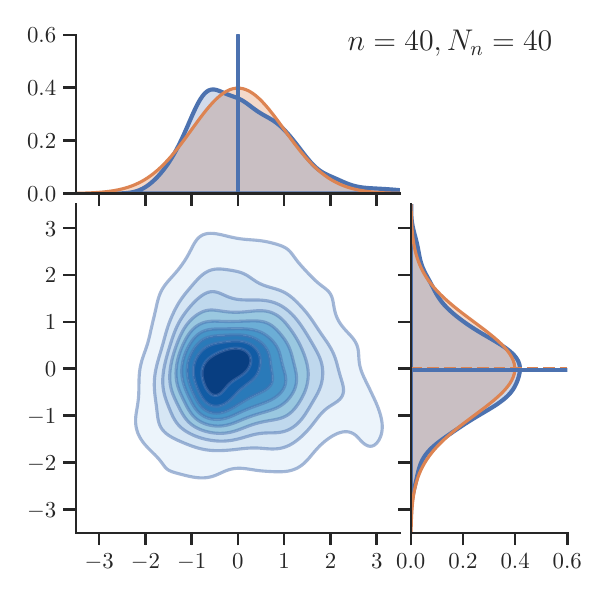}
    \includegraphics[width = 0.3\linewidth]{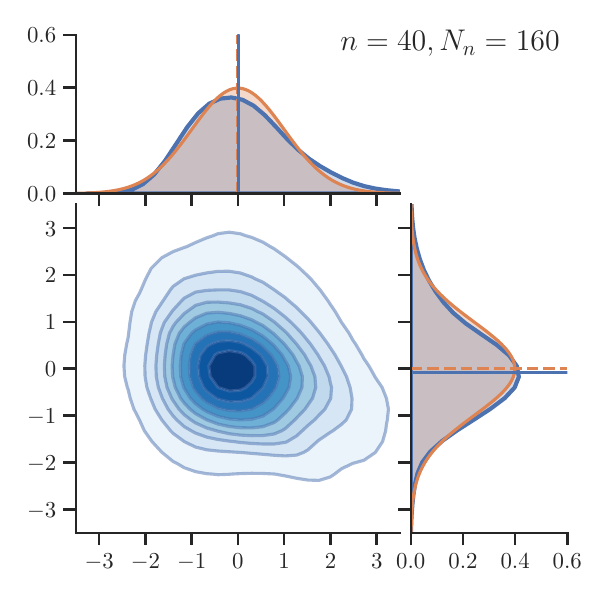}
    
    \caption{2D KDE plots of $\beta = \sqrt{n} \asympConMat^{1/2} (\thetaMLapproxNn - \theta_0) \in \mathbb{R}^2$ for Examples~\ref{example:sampling}-\ref{example:Bernstein} (from top to bottom: sampling, random sampling with time-independent or time-varying samples, Bernstein), and for different values of $n \geq 20$ and $N_n \geq 40$ (results shown by rows). As shown in Corollary~\ref{theo:normalityApprox2}, $\beta \sim \mathcal{N}(0, I_2)$ as $n, N_n \to \infty$. For the case of random sampling with time-varying samples (third row), due to the condition $\frac{N_n}{\log n} \to \infty$, we let $N_n = 160$. The KDE (blue) and asymptotic (orange) marginal distributions of $\beta$ are shown in the $x-y$ axes, with dashed lines indicating the median values.}
    \label{fig:normalityApprox}
\end{figure}

Figure~\ref{fig:normalityApprox} presents 2D KDE plots for $\beta = \sqrt{n} \asympConMat^{1/2} (\thetaMLapproxNn - \theta_0) \in \mathbb{R}^2$. As shown in Corollary~\ref{theo:normalityApprox2}, the random vector $\beta$ is expected to follow an asymptotic distribution $\mathcal{N}(0, I_2)$ as $n, N_n \to \infty$. The blue areas represent the 2D confidence ensembles at varying levels of confidence for $\beta$. Additionally, we display the KDE plots of the marginal distributions of $\beta$. We must clarify that these marginals do not correspond directly to the estimated parameters $\widehat{\theta}_1$ and $\widehat{\theta}_2$, but rather to centered linear combinations of them. Consistently with the results in Figure~\ref{fig:consistencyApprox}, we observe that, as $n$ and $N_n$ increase, the marginals become closer to the standard Gaussian distribution, thus supporting Theorem~\ref{theo:normalityApprox} and Corollary~\ref{theo:normalityApprox2} related to the asymptotic normality of the ML estimator. Although the marginals indeed match the asymptotic distributions, this is not clear from the 2D KDE plots. We believe that this discrepancy is likely due to numerical issues stemming from the estimation method used by the 2D KDE function available in the Python toolbox Seaborn. We anticipate that considering tens of thousands of random replicates could yield better results, closer to the asymptotic distribution $\mathcal{N}(0, I_2)$.

\section{Conclusions}
\label{sec:conclusions}
To the best of our knowledge, we have introduced the first general convergence guarantees for the maximum likelihood (ML) estimator of a Gaussian process (GP) with functional inputs. Within the framework of increasing-domain asymptotics, we have shown both the consistency and normality of the estimator. Moreover, we extend these findings to encompass scenarios accounting for approximation errors in the inputs. These extensions offer valuable insights into certifying the robustness of practical implementations that rely on approximations obtained through conventional sampling methods or projections onto a fixed functional basis. Through several analytical and numerical examples, we illustrate the validity of the asymptotic guarantees.
\\
\\
A natural extension of this work includes convergence results within the fixed-domain asymptotic setting. This will require specific derivation on a case-by-case basis, tailored to particular covariance functions. %Furthermore, theoretical guarantees for the GP framework proposed in~\cite{Muehlenstaedt2017CompExpFunInputs}, which accounts for both scalar-valued and functional inputs, can be investigated.

\paragraph{Acknowledgments.} This work was supported by the National Institute for Mathematical Sciences and Interactions (INSMI), CNRS, as part of the PEPS JCJC 2023 call; and the project GAME (ANR-23-CE46-0007) of the French National Research Agency (ANR).

\clearpage

\section*{Appendix A: Conditions for isotropic covariance functions}
\label{app:CondIsotropy}
This appendix provides the adaptions of Conditions~\ref{cond:covFunctCondition}, \ref{cond:dcovFunctCondition}, \ref{cond:d3covFunctCondition}, \ref{cond:derivative:t}, \ref{cond:derivative:t:der:theta} and~\ref{cond:derivative3:t:der:theta} to the case of isotropic covariance functions.

\begin{Condition}
\label{cond:covFunctConditionIso}
Suppose that the covariance structure given by \eqref{eq:kernel} is both stationary and isotropic, that is we can write the kernels $K_{\theta}$ as
\begin{equation*}
		\forall \theta \in \Theta, \forall f,g \in \functionSpace, \ K_\theta(f, g) = F_\theta(\|f - g\|),
	\end{equation*}
	with $F_\theta: \realset{+} \to \realset{}$.
    Furthermore, suppose that Conditions~\ref{cond:distant} and \ref{cond:orthormalBasis} hold, and that there exists a fixed $\gamma_0 > d_0$, such that for all $t \ge 0$, it holds 
	\begin{equation*}
		\sup_{\theta \in \Theta} |F_\theta(t)| \ \triangleleft \ \frac{1}{1+t^{\gamma_0} }.
	\end{equation*}
\end{Condition}

We can also rewrite Conditions \ref{cond:dcovFunctCondition} and~\ref{cond:d3covFunctCondition} in the isotropic case as follows.
\begin{Condition}
	\label{cond:dcovFunctConditionIso}
	Assume Conditions~\ref{cond:distant} and \ref{cond:orthormalBasis} holds. For all $t \geq 0$, $F_\theta(t)$ is continuously differentiable with respect to $\theta$ and, for some $\gamma_1 > d_0$, we have
	\begin{equation*}
		\sup_{\theta \in \Theta} \max_{i = 1, \ldots, p} \left|\frac{\partial F_\theta}{\partial \theta_i}(t)\right| \ \triangleleft \ \frac{1}{1+t^{\gamma_1} }.
	\end{equation*}
\end{Condition}

    \begin{Condition}
	\label{cond:d3covFunctConditionIso}	Assume Conditions~\ref{cond:distant} and \ref{cond:orthormalBasis} holds. For all $t \geq 0$, $\theta \mapsto F_\theta(t)$ is three times continuously differentiable and we have, for $q \in \{2,3\}$, $i_1, \ldots, i_q \in \{1, \ldots, p\}$, and for some $\gamma_2 > d_0$,
	\begin{equation*}
		\sup_{\theta \in \Theta} \left|\frac{\partial^qF_\theta}{\partial \theta_{i_1} \cdots\partial \theta_{i_q}} (t)\right| \ \triangleleft \ \frac{1}{1+t^{\gamma_2} }.
	\end{equation*}
	\end{Condition}
	%Note that the previous condition allows us to consider the case of anisotropic. However, if we limit ourselves to the isotropic case, then Condition \ref{cond:covFunctCondition} can be rewritten as follows.

\smallskip

The isotropic equivalents to Conditions \ref{cond:derivative:t}, \ref{cond:derivative:t:der:theta}, and \ref{cond:derivative3:t:der:theta} are given below.
\begin{Condition}
\label{cond:derivative:t:Iso}
Keeping the same constant $\gamma_0$ as in Condition \ref{cond:covFunctCondition}, there exists another constant $\rho_0 > 0$ such that for all $t_1,t_2 \geq 0$ such that $0 < |t_1-t_2| \le \rho_0$, we have
	\begin{equation*}
	    \sup_{\theta \in \Theta} 
	\left|\frac{
		F_{\theta}(t_1) - F_{\theta}(t_2)
	}{t_1 - t_2} \right| 
	\ \triangleleft \ \frac{1}{1+\min(t_1,t_2)^{\gamma_0} }.
	\end{equation*}
\end{Condition}

%While Condition \ref{cond:derivative:t:der:theta} can be expressed as:
\begin{Condition} \label{cond:derivative:t:der:theta:Iso}
	Assume that Conditions \ref{cond:dcovFunctCondition} holds with the same constant $\gamma_1 > 0$. Then there exists another constant $\rho_1 > 0$ such that for all $i \in \{1, \ldots, p\}$, and for all $t_1,t_2 \geq 0$ such that $0 < |t_1-t_2| \leq \rho_1$, we have
	\begin{equation*}
		\sup_{\theta \in \Theta}
		\frac{1}{|t_1-t_2|}
		\left|
		\frac{\partial F_\theta}{\partial \theta_{i}} (t_1)
		-
		\frac{\partial F_\theta}{\partial \theta_{i}}(t_2)
		\right| \ \triangleleft \ \frac{1}{1+\min(t_1,t_2)^{\gamma_1} }.
	\end{equation*}    
\end{Condition}

\begin{Condition} \label{cond:derivative3:t:der:theta:Iso}
	Assume that Condition \ref{cond:d3covFunctCondition} holds with the same constant $\gamma_2 > 0$. Then there exists another constant $\rho_2 > 0$ such that for all $q \in \{2,3\}$, for all $i_1, \ldots, i_q \in \{1, \ldots, p\}$, and for all $t_1,t_2 \geq 0$ such that $0 < |t_1-t_2| \leq \rho_2$, we have
	\begin{equation*}
		\sup_{\theta \in \Theta}
		\frac{1}{|t_1-t_2|}
		\left|
		\frac{\partial^q F_\theta}{\partial \theta_{i_1}\cdots\partial \theta_{i_q}} (t_1)
		-
		\frac{\partial^q F_\theta}{\partial \theta_{i_1}\cdots\partial \theta_{i_q}}(t_2)
		\right| \ \triangleleft \ \frac{1}{1+\min(t_1,t_2)^{\gamma_2} }.
	\end{equation*}    
\end{Condition}
\newpage
\section*{Appendix B: Proof of Lemmas in Section~\ref{sec:asympProperties}}
\label{app:proofLemmas}
\begin{proof}[Proof of Lemma~\ref{lemma:A}]
Let $g$, $(x_i)_{i \in \mathbb{N}^*}$ and $n$ be like in the lemma. Then, through a partitioning argument, we get
	\begin{align*}
		\sum_{i=1}^{n} g(x_i) 
		= \sum_{k=1}^{\infty} \sum_{i \in \{i \ : \ x_i \in [k-1,k) \}} g(x_i)
		\leq \sum_{k=1}^{\infty} \nu_k g({k-1}).
	\end{align*}
The last inequality holds since $g$ is non-increasing.
\end{proof}

\begin{proof}[Proof of Lemma~\ref{lemma:S}]
	%Let $f \in \functionSpace$. 
	Let $n \in \mathbb{N}^*$. From Condition~\ref{cond:covFunctCondition}, we have
	\begin{align*}
		S_j
		&= \sup_{\theta \in \Theta} \sum_{i=1}^{n} |K_\theta(f_j, f_i)|
		\ \triangleleft \ \sum_{i=1}^{n}{\frac{1}{1+\| f_{j} - f_{i}\|^{\gamma_0}}},
	\end{align*}
	for all $j \in \{1,\ldots,n\}$. From Condition~\ref{cond:orthormalBasis} and Parseval's identity, we get that, for all $k,\ell \in \{1,\ldots,n\}$,
    \begin{align}
    \label{ineq:approx}
    \left\|\left( f_{k,d_0} - f_{\ell,d_0}\right) - \left(f_{k} - f_{\ell}\right) \right\| \leq 2\Bigg(\sum_{m=d_0+1}^\infty s_m^2\Bigg)^{\frac{1}{2}} =: \delta.
    \end{align}
    However, we have the following inequality for all $x \ge 2\delta$
    \begin{align}
    \label{ineq:poly_decrease}
    \frac{C}{1+(x-\delta)^{\gamma_0}} \le \frac{1}{1+x^{\gamma_0}},
    \end{align} 
    with $C = \min\left(\frac{1+\delta^{\gamma_0}}{1+(2\delta)^{\gamma_0}}, \left(\frac{1}{2}\right)^{\gamma _0-1}\right) < 1$. Thus using this inequality as well as \eqref{ineq:approx} and Lemma~\ref{lemma:existd}, we get that
	\begin{align*}
		S_j
		\ \triangleleft \ \sum_{i=1}^{n}{\frac{1}{1+\| f_{j,d_0} - f_{i,d_0}\|^{\gamma_0}}}.
	\end{align*}
 %\color{green}
  %Recall the notation	$f_{j,d} = \sum_{k=1}^d \langle f_j, e_k\rangle e_k$ and let $i_0 \in \argmin_{k=1,\ldots,n}\|f_{j,d_0} - f_{k,d_0}\|$. Then, for all $i \in \{1,\ldots,n\}$, 
%	\begin{equation*}
%		\|f_{j,d_0} - f_{i,d_0}\| \ \ge \|f_{j,d_0} - f_{i_0,d_0}\|.
%	\end{equation*}
%	From the triangular inequality, we derive
%	\begin{align*}
%		2\|f_{j,d_0} - f_{i,d_0}\|
%		&\ge \|f_{i,d_0} - f_{j,d_0}\| + \|f_{j,d_0} - f_{i_0,d_0}\|
%		\\
%		&\ge \|f_{i,d_0} - f_{i_0,d_0}\|.
%	\end{align*}
%	Hence
%	\[
%	\|f_{j,d_0} - f_{i,d_0}\| \ \ge \frac{\|f_{i,d_0} - f_{i_0,d_0}\|}{2}.
%	\]
%	%Since $\|f_{j,d} - f_{i,d}\| \ge \frac{\|f_{i,d} - f_{i_0,d}\|}{2}$, then
%	From this inequality, it follows that
 %   \begin{align*}
  %      \sum_{i=1}^{n}{\frac{1}{1+\| f_{j,d_0} - f_{i,d_0}\|^{\gamma_0}}} 
   %     \le \sum_{i=1}^{n}{\frac{1}{1+\left(\| f_{j,d_0} - f_{i_0,d_0}\|/2\right)^{\gamma_0}}}.
   % \end{align*}
    % \color{black}
	Applying Lemma~\ref{lemma:A}, we get
    \begin{align*}
        \sum_{i=1}^{n}{\frac{1}{1+\| f_{j,d_0} - f_{i,d_0}\|^{\gamma_0}}} 
        & \leq \sum_{k=1}^{\infty}   \frac{\nu_k}{1+(k-1)^{\gamma_0}}
        \\
        & \triangleleft \frac{1}{\Delta ^{d_0}} \sum_{k=1}^{\infty}\frac{k^{d_0-1}}{1+(k-1)^{\gamma_0}}.
    \end{align*}
	with $\nu_k = \#\{i: \|f_{i,d_0} - f_{j,d_0}\| \in [k-1,k)\} \triangleleft \frac{k^{d_0-1}}{\Delta^{d_0}}$ according to Condition~\ref{cond:distant}. Since $\gamma _0 > d_0$,
	\begin{equation*}
        S_j \ \triangleleft \frac{1}{\Delta ^{d_0}}\ \sum_{k=1}^{\infty}\frac{k^{d_0-1}}{1+(k-1)^{\gamma_0}} < \infty.
	\end{equation*}
	%\[
	%    \sup _{\theta} \max _{j=1,\ldots,n}S \le \sup _{\theta} \max _{j=1,\ldots,n}\sum_{i=1}^{n}{e^{-\gamma \| f_{j,d} - f_{i,d}\|}} < \infty.
	%\]
\end{proof}

\begin{proof}[Proof of Lemma~\ref{lemma:S'}]
	Let $n \in \mathbb{N}^*$, then from Condition~\ref{cond:dcovFunctCondition}, we have
	\begin{align*}
		\widetilde{S}_j
		&= \sup_{\theta \in \Theta} \sum_{i=1}^{n} \left|\frac{\partial K_\theta}{\partial \theta_k} (f_j, f_i)\right|\ \triangleleft \ \sum_{i=1}^{n}{\frac{1}{1+\| f_{j} - f_{i}\|^{\gamma_1}}},
	\end{align*}
	for all $j \in \{1,\ldots,n\}$. We can conclude by using similar arguments as in the proof of Lemma~\ref{lemma:S}.
\end{proof}

\begin{proof}[Proof of Lemma~\ref{lemma:lambdaMax}]
Using Lemma~\ref{lemma:S} as well as the Gerschgorin's circle theorem, we can show that
	\begin{equation*}
		\sup_{\theta \in \Theta} \lambda_{\max} (R_\theta) \leq \sup_{\theta \in \Theta} \max_{j=1,\ldots,n} \sum_{i=1}^{n} |K_\theta(f_j, f_i)|
	\end{equation*}
	is bounded as $n \to \infty$. For the second part, it suffices to replace Lemma~\ref{lemma:S} by Lemma~\ref{lemma:S'} and argue in the same manner as before to get that 
	\begin{equation*}
		\sup_{\theta \in \Theta} \max_{i = 1, \ldots, p} \lambda_{\max} \left(\frac{\partial R}{\partial \theta_i}(\theta) \right)
	\end{equation*}
	is bounded as $n \to \infty$.
\end{proof}

\begin{proof}[Proof of Lemma~\ref{lemma:consitencyLike}]
	The function $\theta \mapsto \negloglikeFun_{n}(\theta)$ is $C^1$, note that
	$$
	M
    :=
    \sup_{\theta \in \Theta}
 \left|\frac{\partial\left(\negloglikeFun_{n}-\mathbb{E}\left[\negloglikeFun_{n}(\cdot)\right]\right)}{\partial \theta_i}(\theta)\right| 
    = O_P(1).
 %\le 2 M.
	$$
	Let $\epsilon > 0$ and consider $\widetilde{\theta}_1, \ldots, \widetilde{\theta}_k \in \Theta$ such that all $(B(\widetilde{\theta}_i, \epsilon))_i$ constitute an open cover.
	Let $\theta \in \Theta$ and denote by $\widetilde{\theta}_{i_0}$ the closest $\widetilde{\theta}_i$ from $\theta$. Then, by the mean value inequality, we have
	$$
	\left|\negloglikeFun_{n}(\theta) - \mathbb{E}\left[\negloglikeFun_{n}(\theta)\right] - \negloglikeFun_{n}(\widetilde{\theta}_{i_0}) + \mathbb{E}\left[\negloglikeFun_{n}(\widetilde{\theta}_{i_0})\right]\right| \le 2M\left|\theta - \widetilde{\theta}_{i_0}\right|.
	$$
	Thus
	$$
	\left|\negloglikeFun_{n}(\theta) - \mathbb{E}\left[\negloglikeFun_{n}(\theta)\right]\right| \le 2M\left|\theta - \widetilde{\theta}_{i_0}\right| + \left|\negloglikeFun_{n}(\widetilde{\theta}_j) - \mathbb{E}\left[\negloglikeFun_{n}(\widetilde{\theta}_{i_0})\right]\right|.
	$$
	However $|\negloglikeFun_{n}(\widetilde{\theta}_{i_0}) - \mathbb{E}[\negloglikeFun_{n}(\widetilde{\theta}_{i_0})]| = o_P(1)$ as well as $|\theta - \widetilde{\theta}_{i_0}| \le \epsilon$ with $\epsilon$ arbitrarily chosen. This concludes the proof.
\end{proof}

\begin{proof}[Proof of Lemma~\ref{lemma:S''}]
    Let $n \in \mathbb{N}^*$, $q \in \{2,3\}$ and $i_1, \ldots, i_q \in \{1, \ldots, p\}$, then from Condition \ref{cond:d3covFunctCondition}, we have 
    \begin{align*}
		\sup_{\theta \in \Theta} \sum_{i=1}^{n} \left|\frac{\partial^q K_\theta}{\partial \theta_{i_1}\cdots \partial \theta_{i_q}} (f_j, f_i)\right|\ \triangleleft \ \sum_{i=1}^{n}{\frac{1}{1+\| f_{j} - f_{i}\|^{\gamma_2}}},
	\end{align*}
	for all $j \in \{1,\ldots,n\}$. We can conclude by using similar arguments as in the proof of Lemma~\ref{lemma:S}.
\end{proof}

\begin{proof}[Proof of Lemma~\ref{lemma:lambdaMaxDeriatives}]
	As in the proof of Lemma~\ref{lemma:lambdaMax}, we obtain this result\\ through the appropriate use of Gershgorin's circle theorem as well as Lemma~\ref{lemma:S''}.
\end{proof}

\newpage
\section*{Appendix C: Proof of Lemmas in Section~\ref{sec:asympPropertiesApprox}}
\label{app:proofLemmasApprox}

\begin{proof}[Proof of Lemma~\ref{lemma:lambdaMaxApprox}]
    According to Gerschgorin's circle theorem, we have for any $\theta \in \Theta$
	\begin{align*}
		\lambda_{\max}\left(R_\theta-R_{N_n, \theta}\right)
		&\leq 
		\max_{i=1, \ldots, n} \sum_{j = 1}^{n} \left|(R_\theta)_{i,j}-(R_{N_n, \theta})_{i,j}\right|
		\\
		&= \max_{i=1, \ldots, n} \sum_{j = 1}^{n} \left|K_{\theta}(f_i,f_j) - K_{\theta}(f_i^{N_n},f_j^{N_n})\right|.
	\end{align*}
	Let $b_n = \max_{i=1}^n \|f_i-f_{i}^{N_n}\|$. Then
	\begin{equation*}
		\|(f_i-f_j) - (f_i^{N_n}-f_j^{N_n}) \|
		\leq \|f_i-f_i^{N_n}\| + \|f_j-f_j^{N_n}\|
		\leq 2 b_n.
	\end{equation*}
	Since $\lim\limits_{n \to \infty} b_n = 0$ from Condition~\ref{cond:approx:inputs}, from Condition~\ref{cond:derivative:t}, as well as the inequality in \eqref{ineq:poly_decrease}, for $n$ large enough
	\begin{align*}
		\sup_{\theta \in \Theta}\lambda_{\max}\left(R_\theta-R_{N_n, \theta}\right)
		&\ \triangleleft \ b_n\max_{i=1, \ldots, n} \sum_{j=1}^{n} \frac{1}{1+\min(\|f_i-f_j\|, \|f_i^{N_n}-f_j^{N_n}\|)^{\gamma_0}}
		%\\
		%&\leq \max_{i=1, \ldots, n} \sum_{j=1}^{n} 2 b_N C e^{-\gamma [\|f_i-f_j\|-2b_N]}
		\\
		&\ \triangleleft \ b_n \max_{i=1, \ldots, n}  \sum_{j=1}^{n}  \frac{1}{1+\|f_i-f_j\|^{\gamma_0}}.
	\end{align*}
	We recall that $\|f_{j,d_0} - f_{i,d_0}\|\geq \Delta$ and $\frac{1}{1+|f_i-f_j\|^{\gamma_0}} \triangleleft \frac{1}{1+\|f_{j,d_0} - f_{i,d_0}\|^{\gamma_0}}$. Therefore, applying Lemma~\ref{lemma:A}, we get
	\begin{align*}
		\sup_{\theta \in \Theta}\lambda_{\max}\left(R_\theta-R_{N_n, \theta}\right)
		&\ \triangleleft \ b_n \max_{i=1, \ldots, n}  \sum_{j=1}^{n}  \frac{1}{1+\|f_i-f_j\|^{\gamma_0}}
		\\
		&\ \triangleleft \ \frac{b_n }{\Delta^{d_0}} \sum_{k=1}^{\infty} \frac{k^{d_0-1}}{1+(k-1)^{\gamma_0}}.
	\end{align*}
	Since $\lim\limits_{n \to \infty} b_n = 0$, then the first part of the lemma is proven. The second part can be established using similar arguments.
\end{proof}

\begin{proof}[Proof of Lemma~\ref{lemma:convConMat}]
Let $i,j \in \{1, \ldots p\}$ and $n \in \mathbb{N}^*$. Remark that all the following matrices are symmetric.
\begin{align*}
    A &:= R_{\theta_0}^{-1}, 
      &B &:= \frac{\partial R}{\partial \theta_i}(\theta_0),
     &C &:= \frac{\partial R}{\partial \theta_j}(\theta_0),
     \\
    A_n &:= R_{N_n,\theta_0}^{-1},
     &B_n &:= \frac{\partial R_{N_n}}{\partial \theta_i}(\theta_0),
     &C_n &:= \frac{\partial R_{N_n}}{\partial \theta_j}(\theta_0).     
\end{align*}

Using the Cauchy-Schwarz inequality on the space of square matrices endowed with the norm $\|\cdot \|_2$, where we recall that $\|R\|_2^2 = \frac{1}{n}\sum_{i,j=1}^nR_{i,j}^2$, we obtain
\begin{align*}
    \left|(\asympConMatapproxNn - \asympConMat)_{i,j}\right|  = \ & \frac{1}{2n}\bigg|\tr\bigg(R_{N_n,\theta_0}^{-1} \frac{\partial R_{N_n}}{\partial \theta_i}(\theta_0) R_{N_n,\theta_0}^{-1} \frac{\partial R_{N_n}}{\partial \theta_j}(\theta_0)\bigg)\\
    &\hspace{4ex} - \tr\bigg(R_{\theta_0}^{-1} \frac{\partial R}{\partial \theta_i}(\theta_0) R_{\theta_0}^{-1} \frac{\partial R}{\partial \theta_j}((\theta_0)\bigg)\bigg|\\
     =  \ &\frac{1}{2n}\bigg|\tr\left(A_nB_nA_nC_n\right) - \tr\left(ABAC\right)\bigg|\\
     =  \ &\frac{1}{2n}\bigg|\tr\bigg((A_n -A)B_nA_nC_n + A(B_n-B)A_nC_n\\
    &\hspace{9ex} +AB(A_n-A)C_n + ABA(C_n-C)\bigg)\bigg|
    \\
     \le  \ &\|A_n -A\|_2\|B_nA_nC_n\|_2 + \|B_n -B\|_2\|A_nC_nA\|_2\\
    &+ \|A_n -A\|_2\|C_nAB\|_2 + \|C_n -C\|_2\|ABA\|_2.
\end{align*}
 However, we have
\[
\|A_n -A\|_2 = \|R_{N_n,\theta}^{-1} - R_{\theta}^{-1}\|_2 \xrightarrow[n \to \infty]{} 0,
\]
as well as
\[
\|B_n -B\|_2\xrightarrow[n \to \infty]{} 0 \quad \textrm{and} \quad \|C_n -C\|_2\xrightarrow[n \to \infty]{} 0.
\]
Moreover the quantities $\|B_nA_nC_n\|_2$, $\|A_nC_nA\|_2$, $\|C_nAB\|_2$ and $\|ABA\|_2$ are bounded in $n$. Thus
\[
\asympConMatapproxNn - \asympConMat \xrightarrow[n \to \infty]{} 0.
%= o_{\textrm{a.s.}}\left(\frac{1}{n}\right).
\]
\end{proof}
\newpage
\bibliographystyle{plain}
\bibliography{references}       % Bibliography file (usually '*.bib')

\end{document}